\documentclass[pdflatex,sn-mathphys-num]{sn-jnl}

\usepackage{graphicx}%
\usepackage{multirow}%
\usepackage{amsmath,amssymb,amsfonts}%
\usepackage{amsthm}%
\usepackage{mathrsfs}%
\usepackage[title]{appendix}%
\usepackage{xcolor}%
\usepackage{textcomp}%
\usepackage{manyfoot}%
\usepackage{booktabs}%
\usepackage{algorithm,algpseudocode}%
\usepackage{listings}%

\usepackage{latexsym}
\usepackage{dsfont}
\usepackage{amsfonts}
\usepackage{amssymb}
\usepackage{amsmath}
\usepackage{mathtools}
\usepackage{enumitem}
\usepackage{changes}
\usepackage{placeins}
\usepackage{pgf,tikz}
\usepackage{pgfplots}
\usepgfplotslibrary{colorbrewer}
\usetikzlibrary{plotmarks}
\usetikzlibrary{angles}
\usetikzlibrary{arrows}
\usetikzlibrary{arrows.meta}
\usetikzlibrary{calc}
\usetikzlibrary{decorations.markings}
\usetikzlibrary{decorations.pathmorphing}
\usetikzlibrary{external}
\usetikzlibrary{math}
\usetikzlibrary{patterns}
\usepackage{soul}

\newcommand{\reg}{D}
\newcommand{\regg}{d}
\newcommand{\regOp}{\mathcal{\reg}_\varepsilon}
\newcommand{\R}{\mathbb R}
\newcommand{\N}{\mathbb N}
\newcommand{\pre}{\mathcal{F}}

\newcommand{\corr}{K}
\newcommand{\id}{\mathbb I}
\newcommand{\New}{Newton} 
\newcommand{\Newc}{Newton$_{\varepsilon}$} 
\newcommand{\NlinRAS}{Newton$_{\text{RAS}}$} 
\newcommand{\NlinRASc}{Newton$_{\text{RAS},\varepsilon}$} 
\newcommand{\RASPEN}{RASPEN}
\newcommand{\RASPENc}{RASPEN$_\varepsilon$}

\DeclareMathOperator{\proj}{proj}

\theoremstyle{thmstyleone}%
\newtheorem{theorem}{Theorem}
\newtheorem{lemma}{Lemma}
\newtheorem{corollary}{Corollary}

\theoremstyle{thmstyletwo}%
\newtheorem{remark}{Remark}%

\theoremstyle{thmstylethree}%
\newtheorem{assumption}{Assumption}

\begin{document}

\title[Article Title]{Solving Semi-Linear Elliptic Optimal Control Problems with $L^1$-Cost via Regularization and RAS-Preconditioned Newton Methods}


\author[1]{\fnm{Gabriele} \sur{Ciaramella}}\email{gabriele.ciaramella@polimi.it}

\author*[2]{\fnm{Michael} \sur{Kartmann}}\email{michael.kartmann@uni-konstanz.de}

\author[3]{\fnm{Georg} \sur{Müller}}\email{georg.mueller@uni-heidelberg.de}

\affil[1]{\orgdiv{MOX Lab, Dipartimento di Matematico}, \orgname{ Politecnico di Milano}, \orgaddress{\street{Piazza Leonardo da Vinci 32}, \city{Milano}, \postcode{20133}, \country{Italy}}}

\affil*[2]{\orgdiv{Departement of Mathematics}, \orgname{University of Konstanz}, \orgaddress{\street{Universitätsstraße 10}, \city{Konstanz}, \postcode{78457}, \country{Germany}}}

\affil[3]{\orgdiv{Interdisciplinary Center for Scientific Computing}, \orgname{Heidelberg University}, \orgaddress{ \city{Heidelberg}, \postcode{69120},  \country{Germany}}}

\abstract{We present a new parallel computational framework for the efficient solution of a class of $L^2$/$L^1$-regularized optimal control problems governed by semi-linear elliptic partial differential equations (PDEs).
The main difficulty in solving this type of problem is the nonlinearity and non-smoothness of the $L^1$-term in the cost functional, which we address by employing a combination of several tools.
First, we approximate the non-differentiable projection operator appearing in the
optimality system by an appropriately chosen regularized operator and establish convergence of the resulting system's solutions.
Second, we apply a continuation strategy to control the regularization parameter to improve the behavior of (damped) Newton methods.
Third, we combine Newton's method with a domain-decomposition-based
nonlinear preconditioning, which improves its robustness properties and allows for parallelization.
The efficiency of the proposed numerical framework is demonstrated by extensive numerical experiments.}




\maketitle

	\section{Introduction}
In this paper, we combine a smoothing-continuation technique and do\-main\--de\-com\-po\-si\-tion\--based nonlinear preconditioning of Newton's method to obtain a novel, robust, efficient computational framework for finding stationary points of the semi-linear elliptic optimal control problem
\begin{subequations}
	\label{eq:P}
	\begin{alignat}{1}
		\label{eq:P_cost}
		\text{minimize }
		{}
		&
		J(y,u)
		\coloneqq
		{}
		\frac{1}{2}\big\|y- y_d\big\|_{L^2}^2
		+
		\frac{\nu}{2}\big\|u\big\|_{L^2}^2
		+
		\mu \big\|u\big\|_{L^1},
		\\
		\label{eq:P_PDE}
		\text{s.t. }
		{}
		&
		(y,u)\in H_0^1(\Omega) \times L^2(\Omega)
		\\
		{}
		&
		A y + \varphi(y)= f + u \; \text{ in } \; H^{-1}(\Omega).
	\end{alignat}
\end{subequations}
Here $\Omega \subseteq \R^n$ is a bounded Lipschitz domain, $y \in H_0^1(\Omega)$ and $u \in L^2(\Omega)$ are the state- and control function, respectively, $y_d\in L^2(\Omega)$ is a desired state,
the operator $A\colon H^1_0(\Omega) \rightarrow H^{-1}(\Omega)$ is a linear, self-adjoint, elliptic differential operator (in weak form)
and $\varphi$ is a nonlinear real function.
The parameters $\nu,\mu > 0$ act as weights for the regularization terms in the cost functional.

This class of problems is particularly interesting because of the non-smooth $L^1$-type regularization term in the cost function, which promotes sparsity in the optimal controls' support but requires solution techniques that are able to cope with its non-differentiability.
Sparsity in the control solution is desirable, e.g., in actuator- and sensor placement problems, see, e.g., the introduction of \cite{Stadler:2009:1}, where optimality conditions and the structure of optimal solutions for a box-constrained, linear(elliptic)-quadratic problem with $L^1$-type cost functional are examined.
Other relevant works considering problems with $L^1$-type cost functional include \cite{WachsmuthWachsmuth:2011:3}, where advanced computational aspects, including convergence results and error estimates, are discussed.
The authors of \cite{CasasHerzogWachsmuth:2012:2} derive a-priori discretization-error estimates for problems with $L^1$-cost and \emph{semi-linear} elliptic PDEs.
Additionally, $L^1$-type terms (often applied to the gradient of a function) frequently appear in image denoising and impainting problems, see \cite{RudinOsherFatemi:1992:1, Getreuer:2012:1, GoldsteinOsher:2009:1, ChanTai:2003:1, HerrmannHerzogKroenerSchmidtVidalNunez:2018:1} and in quantum control problems, see \cite{CB:2016:1,CB:2016:2}.
In multiobjective optimization, a problem governed by a parabolic, semi-linear PDE constraint and $L^1$-type non-smoothness in the cost functional is treated by set-oriented methods in \cite{BeermannDellnitzPeitzVolkwein:2017:2}.
The authors of \cite{BiekerGebkenPeitz:2022:1} apply a continuation method (in a different sense than it is understood in our work) to a bicriterial optimization problem, where the $L^1$-regularization is one of the objectives to be minimized.
Other continuation approaches for optimal control problems can be found in \cite{WachsmuthWachsmuth:2011:3,CBDW2015}. 
Iterative methods such as Bregman iterations (\cite{Bregman1967}, see also \cite{GoldsteinOsher:2009:1}) or the Primal-Dual Hybrid Gradient (PDHG) method \cite{ChambollePock:2011:1} offer robust and efficient first-order alternatives to Newton-type methods for $L^1$-regularized problems. 
In turn, the nonlinearly preconditioned Newton approach presented in this paper benefits from the use of second-order information and improves standard Newton methods' convergence, and offers natural parallelization.
Our analysis of problem \eqref{eq:P} is based on the refined optimality conditions in \cite{CasasHerzogWachsmuth:2012:2}.
While standard necessary first-order optimality conditions for problem \eqref{eq:P} can easily be derived by applying Clarke's subdifferential calculus, the authors of \cite{CasasHerzogWachsmuth:2012:2} derived an explicit, non-smooth (but Lipschitz-continuous) representation of the subderivative that corresponds to the non-smooth $L^1$-part of the cost functional.
Applying this technique to problem \eqref{eq:P}, we obtain the first-order necessary optimality system
\begin{subequations}\label{eq:non-smooth_OS}
	\begin{align}
		A \bar{y}+\varphi(\bar{y})&=f+\bar{u} &&\text{in }   H^{-1}(\Omega), \\
		\Big(A +\varphi'(\bar{y})\Big)\bar{p}&=\bar{y}-y_d &&\text{in }   H^{-1}(\Omega), \\
		\bar{p}+\nu \bar{u}+\mu \proj _{[-1,1]}\bigg(-\frac{\bar{p}}{\mu}\bigg)&= 0 && \text{in } \Omega,\label{eq:non-smooth_OS:3}
	\end{align}
\end{subequations}
with $\bar u\in L^2(\Omega),\, \bar y,\bar p\in H_0^1(\Omega)$.
The standard approach for solving \eqref{eq:non-smooth_OS} is the application of a damped semi-smooth Newton method.
In contrast, our proposed computational framework is built on a parameter continuation technique (see, e.g., \cite{AllgowerGeorg2003}) for a smoothing parameter combined with an extension of the domain-decomposition nonlinear preconditioning approach initially presented in \cite{DGKW16} for solving elliptic PDEs; see also \cite{Gander:2017,CK:2002,CKY:2001,GuKwok2020}.
Specifically, to improve robustness and numerical performance, we propose to regularize the problem by replacing the non-smooth Nemytskii-type projection operator 
\begin{equation*}
	\proj_{[-1,1]}(x) = \begin{cases}
		x & \text{for } x\in [-1,1],\\
		1 & \text{for } x>1,\\
		-1 & \text{for } x<-1. 
	\end{cases}
\end{equation*}
for $ x\in \R$ in \eqref{eq:non-smooth_OS:3} with a smoothed approximation $P_\varepsilon$ and to solve the smoothed versions of the problem efficiently using preconditioned Newton-Krylov methods as part of a continuation strategy for the smoothing parameter $\varepsilon$, where the subproblems of the continuation strategy are solved using an extension of the Restricted Additive Schwarz Preconditioned Exact Newton method (RASPEN, \cite{DGKW16}), which is the application of Newton's method to the fixed-point equation derived from the nonlinear Restricted Additive Schwarz (RAS) iteration for the regularized first-order system \eqref{eq:non-smooth_OS}.
Nonlinear RAS is a domain decomposition method that computes the solution to a given problem defined on a domain $\Omega$ by iteratively solving smaller subproblems defined on subdomains of $\Omega$, allowing for parallelization across the subdomains.

Note that a similar nonlinear preconditioning approach has been proposed in \cite{CKM2021} for elliptic-PDE-constrained optimization problems and in \cite{CM:2021} for economic parabolic control problems.
However, these approaches are based on a different domain decomposition method using Robin-type transmission conditions applied directly to the non-smooth optimality system.

This work is organized as follows.
In Section \ref{sec:NotAndAss}, we fix the required notation, state the main assumptions used in this work and collect preliminary results including the fundamental first-order optimality system \eqref{eq:non-smooth_OS}.
In Section \ref{sec:reg}, we introduce the regularization of the optimality system and prove that there exist solutions by showing that it corresponds to a necessary first-order optimality system of a solvable smooth optimization problem.
Section \ref{sec:convAna} focuses on the convergence analysis of the regularized systems' solutions to the solution of the original non-smooth system \eqref{eq:non-smooth_OS}.
In Section \ref{sec:FrameworkAndNumerics}, we introduce and extend the RAS and RASPEN preconditioning techniques for systems of PDEs.
Finally, Section \ref{sec:num} investigates and compares the efficiency of the numerical approaches.
Specifically, we examine the influence of introducing combinations of the regularization, the parameter continuation and (non-)linear RAS preconditioning on the performance of solvers on the outer (Newton) and inner (GMRES) level with respect to number of iterations, stability of numbers of iterations and computation time.
A short conclusion of our findings is presented in Section \ref{sec:concl}.

\section{Notation, assumptions and preliminary results}\label{sec:NotAndAss}
As long as the meaning is clear from context, a Nemytskii operator associated to a real function is denoted by the same symbol.
The Nemytskii operator of the nonlinearity $\varphi\colon \mathbb R \to \mathbb R$ in the PDE-constraint is understood to map $L^2$ into itself.
All norms on Hilbert spaces $(H,\langle \cdot, \cdot \rangle)$ are assumed to be induced by the scalar product unless stated otherwise.
The space $H_{0}^1(\Omega)$ is endowed with the inner product ${{\langle u,v\rangle}_{H_0^1}=\int_\Omega\nabla u\cdot\nabla v+ uv\,\mathrm d x}$.
The corresponding dual space is denoted as $H_{0}^1(\Omega)' = H^{-1}(\Omega)$.
When elements in $H_{0}^1(\Omega)$
are interpreted as elements of $H^{-1}(\Omega)$, this always means the Gelfand-type identification via the embedding into $L^2(\Omega)$ and the $L^2$-Riesz mapping.

\begin{assumption}\label{assumption:standing}$\;$ 
	\begin{enumerate}[leftmargin=0.5cm]
		\item The set $\Omega \subseteq\mathbb{R}^n$ for $n\in\{2,3\}$ is a bounded domain with $C^{0,1}$ boundary (see, e.g., \cite[Section 6.2]{GT01}).
		\item The functions $f,y_d$ are in $L^2(\Omega)$ and $\mu, \nu\in\mathbb{R}_{>0}$.
		\item The operator $A\colon H^1_0(\Omega) \rightarrow H^{-1}(\Omega)$ is linear and elliptic with corresponding strong differential form $\mathcal{A}y\coloneqq-\sum\limits_{i,j=1}^{n}\partial_{x_j}(a_{ij}\partial_{x_i}y)+a_0 y$,
		$a_0,a_{ij}\in L^\infty(\Omega),\, a_0\geq 0$ and $a_{ij}=a_{ji}$.
		Moreover, there exists a $C_A>0$ such that
		$\sum\limits_{i,j=1}^{n}a_{ij}(x)\xi_i\xi_j\geq C_A|\xi|^2$ for all $\xi\in\mathbb{R}^n$
		and for a.a.~$x$ in $\Omega$.
		\item The function $\varphi\in C^2(\mathbb{R})$ is monotonically increasing and $\varphi''$ is locally Lipschitz continuous.
	\end{enumerate}
\end{assumption}
Assumption \ref{assumption:standing} guarantees that
there exists a well-defined solution operator to the constraining PDE \eqref{eq:P_PDE}
and its adjoint form.

\begin{lemma}[The solution operator $S$]
	\label{lem:forwardSol}
	For every $u\in L^2(\Omega)$, there exists a unique solution $y=y(u)\in H^1_0(\Omega)\cap L^{\infty}(\Omega) $ of  \eqref{eq:P_PDE}.
	Thus, the map $S\colon L^2(\Omega)\to  H^1_0(\Omega)\cap L^{\infty}(\Omega)$, $ S(u)\coloneqq y(u)$ is well defined and there exist two constants $C,L>0$ such that,
	for all $u,u_1,u_2\in L^2(\Omega)$, it holds that
	\begin{align}
		\big\|S(u)\big\|_{H^1}+\big\|S(u)\big\|_{L^{\infty}}&\leq C \big\|f+u-\varphi(0)\big\|_{L^2},\label{eq:bounded}\\
		\big\|S(u_1)-S(u_2)\big\|_{H^1}+\big\|S(u_1)-S(u_2)\big\|_{L^{\infty}}&\leq L \big\|u_1-u_2\big\|_{L^2}.\label{eq:Lipschitz}
	\end{align}
\end{lemma}
\begin{proof}
	The claim is proved in \cite[Section 4]{Troeltzsch05} for Neumann and Robin boundary conditions. In particular, \cite[Theorems 4.4, 4.5]{Troeltzsch05} show the existence of solutions and the boundedness result \eqref{eq:bounded} for bounded $\varphi$ with $\varphi(0)=0$, and \cite[Theorems 4.7, 4.8]{Troeltzsch05} show that the latter assumptions can be actually dropped to obtain the claim. All results transfer to homogeneous Dirichlet boundary conditions, in which case the bilinear form is naturally coercive. The Lipschitz continuity \eqref{eq:Lipschitz} is proved in \cite[Theorem~4.16]{Troeltzsch05} and also carries over immediately.
\end{proof}
\begin{lemma}[The adjoint problem]
	\label{lem:adjointSol}
	Let $y\in H^1(\Omega)\cap L^\infty(\Omega)$.
	Then for every $v\in L^2(\Omega)$, there exists a unique solution $p=p(v)\in H^1_0(\Omega)\cap L^{\infty}(\Omega) $ to the problem
	\begin{equation}
		\Big(A +\varphi'(y)\Big){p}=v \text{ in } H^{-1}(\Omega).\label{eq:adjoint_v}
	\end{equation}
	Moreover, there exist two constants $C,L>0$, independent of $y$, such that
	\begin{align*}
		\big\|p(v)\big\|_{H^1}+\big\|p(v)\big\|_{L^{\infty}}&\leq C \big\|v\big\|_{L^2},\\
		\big\|p(v_1)-p(v_2)\big\|_{H^1}+\big\|p(v_1)-p(v_2)\big\|_{L^{\infty}}&\leq L \big\|v_1-v_2\big\|_{L^2},
	\end{align*}
	for all $v,v_1,v_2\in L^2(\Omega)$.
\end{lemma}
\begin{proof}
	First, we know that $\varphi'\geq0$. Since $y$ is in $L^\infty(\Omega)$, continuity of $\varphi'$ yields that $\varphi'(y)\in L^\infty(\Omega)$.
	Accordingly, we can define the operator $\widetilde A = A + \varphi'(y)$ and $\widetilde\varphi \equiv 0$ and the adjoint problem \eqref{eq:adjoint_v} is obviously equivalent to $\widetilde A p  + \widetilde \varphi p = v$ and we can proceed analogously to the proof of Lemma \ref{lem:forwardSol}.
	In the proofs of the theorems from \cite{Troeltzsch05}, we notice that the constants can be chosen independently of $y$ because the part of $\varphi'$ can be dropped in any of the estimates due to coercivity of $A$.
\end{proof}

Further, the existence of at least one global minimizer can be obtained using standard arguments, cf.~\cite[Sec 4.4.2]{Troeltzsch05}.
\begin{lemma}[Existence of minimizers]
	\label{lem:minEx}
	Let Assumption \ref{assumption:standing} be satisfied. Then there exists at least one global minimizer for problem \eqref{eq:P}.
\end{lemma}
%

Now, we can state the first-order necessary optimality condition.
\begin{theorem}[First-order optimality system] \label{thm:firstOrderConds}
	Let $\bar{u}\in L^2(\Omega)$ be a local minimizer of \eqref{eq:P}.
	Then there exist unique $\bar{y}, \bar{p} \in H^1_0(\Omega)\cap L^{\infty}(\Omega)$ and a $\bar{\lambda}\in \partial_C \|\bar{u}\|_{L^1(\Omega)}$ (where $\partial_C$ denotes Clarke's generalized differential; see \cite[Sec. 2.1]{Clarke90}) that satisfy
	\begin{equation}\label{eq:non-smooth_OS_lam}
		\begin{aligned}
			A \bar{y}+\varphi(\bar{y})&=f+\bar{u} &&\text{ in }   H^{-1}(\Omega), \\
			\Big(A +\varphi'(\bar{y})\Big)\bar{p}&=\bar{y}-y_d &&\text{ in }   H^{-1}(\Omega), \\
			\bar{p}+\nu \bar{u}+\mu \bar{\lambda}&= 0&& \text{ in } L^2(\Omega).
		\end{aligned}
	\end{equation}
\end{theorem}
\begin{proof}
	The result follows using Clarke's generalized differential to deal with the non-G\^ateaux-differentiability of the function $u\mapsto \|u\|_{L^1(\Omega)}$, and the abstract optimality conditions derived in \cite[Chap. 2]{Clarke90} and \cite{BonnansShapiro}, cf.~\cite{CasasHerzogWachsmuth:2012:2}.
\end{proof}

Because of the inclusion $\bar{\lambda}\in \partial_C \|\bar{u}\|_{L^1(\Omega)}$, the last equation in \eqref{eq:non-smooth_OS_lam} is not easy to be treated numerically.
However, a projection formula for $\bar{\lambda}$ has been proved in \cite[Cor. 3.2]{CasasHerzogWachsmuth:2012:2}
using an explicit representation of the subdifferential $\partial_C \|\bar{u}\|_{L^1(\Omega)}$, which carries over to our setting.
\begin{lemma}[Explicit form and fixed point equation for $\bar{\lambda}$]\label{lemmaproj} Assume the setting of Theorem \ref{thm:firstOrderConds}.
	Then the generalized derivative $\bar \lambda \in \partial_C \|\bar{u}\|_{L^1(\Omega)}$ satisfies
	\begin{equation}\label{eq:lambda_ident}
		\bar{\lambda}=\proj _{[-1,1]}\bigg(-\frac{\bar{p}}{\mu}\bigg)=\proj _{[-1,1]}\bigg(\frac{\nu}{\mu}\bar{u}+\bar{\lambda}\bigg)\in L^\infty(\Omega).
	\end{equation}
\end{lemma}
Here, the second equality is an immediate result of the last condition in \eqref{eq:non-smooth_OS_lam}.
Note that Lemma \ref{lemmaproj} especially implies the uniqueness of the subderivative $\bar{\lambda}$ for the given minimizer $\bar u$.
Moreover, using the last condition in \eqref{eq:non-smooth_OS_lam} yields that $\bar{u}\in L^\infty(\Omega)$.
Furthermore, \eqref{eq:lambda_ident} can be used to
reduce the first-order optimality system \eqref{eq:non-smooth_OS_lam} solely to the state and adjoint variables, i.e., to the system
\begin{equation}\label{eq:non-smooth_OS_reduced}
	\begin{aligned}
		A {y}+\varphi({y})&=f-\frac{1}{\nu}\bigg({p}+\mu \proj _{[-1,1]}\bigg(-\frac{{p}}{\mu}\bigg)\bigg) && \text{ in }   H^{-1}(\Omega), \\
		\Big(A +\varphi'({y})\Big){p}&={y}-y_d &&\text{ in }   H^{-1}(\Omega).
	\end{aligned}
\end{equation}
This is a non-smooth system of coupled PDEs (the non-smoothness being introduced by the projection operator) with unknowns $y,p\in H_0^1(\Omega)$.

\section{Smoothed optimality systems}\label{sec:reg}
The optimality system \eqref{eq:non-smooth_OS_reduced} includes the Nemytskii operator corresponding to the (non-smooth) projection applied to the adjoint state.
Numerically, this non-smoothness and nonlinearity is the main difficulty to deal with.
In this section, we introduce a smoothing approach for the operator $\proj_{[-1,1]}$, show the existence of solutions to the smoothed system, and prove the convergence of solutions of the smoothed system to solutions of the original system \eqref{eq:non-smooth_OS_reduced} as the smoothing parameter tends to zero.

\subsection{Smoothing of the projection operator}\label{sec:reg:non}
\FloatBarrier
We regularize the projection on the real numbers that defines the Nemytskii operator by rewriting $\proj_{[-1,1]}=\frac{1}{2}(|x+1|-|x-1|)$ for $x\in \R$ and applying the $\varepsilon$-shifted square root regularization to the absolute value terms.
This leads to the smooth approximation $P_\varepsilon\colon \R\to\R$ of proj$_{[-1,1]}\colon \R\to\R$ given by
\begin{equation}\label{eq:smoothProj}
	P_\varepsilon(x)\coloneqq \frac{1}{2}\Big( \sqrt{(x+1)^2+\varepsilon}-\sqrt{(x-1)^2+\varepsilon}\Big)
\end{equation}
and its derivative
\begin{equation}\label{eq:smoothProjDer}
	P_{\varepsilon}'(x)=\frac{1}{2}\Bigg( \frac{x+1}{\sqrt{(x+1)^2+\varepsilon}}-\frac{x-1}{\sqrt{(x-1)^2+\varepsilon}}\Bigg)
\end{equation}
for $\varepsilon\ge0$.
Both functions are depicted in Figure \ref{Fig Pe}, and some of their properties are given in the following lemma.
\begin{figure}[t]
	\centering
	\begin{tikzpicture}[
		trim axis left,
		declare function = {
			Pe(\x,\myeps) = 0.5*(sqrt((\x+1)^2+\myeps) - sqrt((\x-1)^2+\myeps);
		},
		]
		\begin{axis}[
			scale only axis,
			height=4cm,
			width=0.45\textwidth,
			x label style = {at = {(axis description cs:1.0, 0.1)}, anchor = north},
			legend style={draw=white!15!black,legend cell align=left,at = {(0.05,0.95)},anchor=north west},
			cycle list={
				{black,mark={},dashed},
				{red, mark=square,mark repeat=50},
				{blue,mark=o,mark repeat=50},
				{teal,mark=diamond,mark repeat=50}
			},
			]
			\addplot+[thick, domain = -4:4, samples = 101]{Pe(x,0)};
			\addlegendentry{$P_0$};
			\pgfplotsinvokeforeach{0.5, 3, 10}{
				\addplot+[thick, domain = -4:4, samples = 101]{Pe(x,#1)};
				\addlegendentry{$P_{#1}$};
			}
		\end{axis}
	\end{tikzpicture}
	\hfill
	\begin{tikzpicture}[
		trim axis left,
		declare function = {
			dPe(\x,\myeps) = 0.5*((x+1)/sqrt((\x+1)^2+\myeps) - (x-1)/sqrt((\x-1)^2+\myeps);
		}
		]
		\begin{axis}[
			scale only axis,
			height=4cm,
			width=0.48\textwidth,
			x label style = {at = {(axis description cs:1.0, 0.1)}, anchor = north},
			legend style={draw=white!15!black,legend cell align=left,at = {(0.05,0.95)},anchor=north west},
			cycle list={
				{black,mark={},dashed},
				{red, mark=square,mark repeat=50},
				{blue,mark=o,mark repeat=50},
				{teal,mark=diamond,mark repeat=50},
			},
			]
			\addplot[black, thick, dashed, mark={}, domain = -4:4, samples = 101]{dPe(x,0)};
			\addlegendentry{$P'_0$};
			\pgfplotsinvokeforeach{0.5, 3, 10}{
				\addplot+[thick, domain = -4:4, samples = 101]{dPe(x,#1)};
				\addlegendentry{$P'_{#1}$};
			}
		\end{axis}
	\end{tikzpicture}
	\caption[Plot Pe]{$P_{\varepsilon}$ and $P'_{\varepsilon}$ for $\varepsilon \in \{0,0.5,3,10\}$.}
	\label{Fig Pe}
\end{figure}
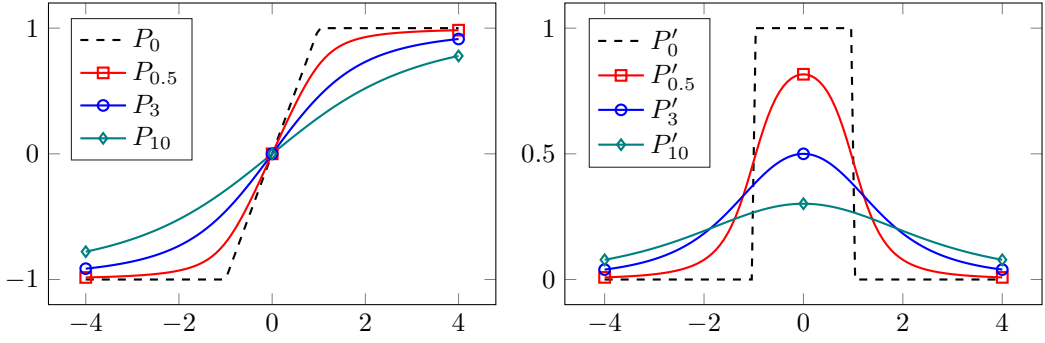

\begin{lemma}[Properties of $P_{\varepsilon}$ and $P'_{\varepsilon}$]\label{lem:PProperties}
	$\;$ 
	\begin{enumerate}[leftmargin=0.5cm]
		
		\item\label{enum:smoothProjDer_i} For every $\varepsilon>0$, $P_{\varepsilon}\in C^{\infty}(\mathbb{R},[-1,1])$ with $\lim\limits_{x\to\pm\infty}P_{\varepsilon}(x)=\pm 1$ and $P_{\varepsilon}$ is strictly monotonically increasing. Additionally, $P_{\varepsilon}(x)\xrightarrow{\varepsilon\searrow0} \proj_{[-1,1]}(x)$ for all $x\in\mathbb{R}$.
		
		\item\label{enum:smoothProjDer_ii} For all $x$ in $\R$, the derivative $P'_{\varepsilon}(x)\in(0,1/\sqrt{1+\varepsilon}]$.
		Moreover,
		\begin{equation*}
			P_{\varepsilon}'(x) \xrightarrow{\varepsilon \searrow 0} \begin{cases}
				1 & \text{for } |x|<1,\\
				1/2 & \text{for } |x|=1,\\
				0 & \text{for } |x|>1.
			\end{cases}
		\end{equation*}
		
		\item\label{enum:smoothProjDer_iii} The associated Nemytskii operators $P_{\varepsilon},P_{\varepsilon}'\colon L^2(\Omega)\to L^{\infty}(\Omega)$ are well defined and $P_\varepsilon$ is globally Lipschitz as an operator mapping $L^2(\Omega)$ into itself.
		\item \label{enum:smoothProjDer_iiii}We have
		\begin{equation*}
			|\proj_{[-1,1]}(x)- P_\varepsilon(x)| \le \sqrt \varepsilon \quad \forall x \in \R
		\end{equation*}
		and 
		\begin{equation}\label{eq:nemistky_estimate}
			\|proj_{[-1,1]}(v)-P_\varepsilon(v)\|_{L^2(\Omega)} \le |\Omega| \sqrt \varepsilon \quad \forall v \in L^{2}(\Omega)
		\end{equation}
		for the respective Nemytskii operators.
	\end{enumerate}
\end{lemma}
\begin{proof}
	
	\underline{Point \ref{enum:smoothProjDer_i}.} The regularity and the pointwise approximation property are given by construction.
	For the boundedness, observe that $P_\varepsilon(x)<0$ ($=0$ or $>0$) if and only if $x<0$ ($x=0$ or $x>0$). For $x>0$, $P_\varepsilon(x)$ is monotonically decreasing in $\varepsilon$ (since
	$\partial_\varepsilon P_\varepsilon(x)=\frac{1}{4}[ ( (x+1)^2+\varepsilon)^{-\frac{1}{2}}-((x-1)^2+\varepsilon)^{-\frac{1}{2}}]<0$).
	Hence $P_\varepsilon(x)\nearrow $ $\proj_{[-1,1]}(x)$ as $\varepsilon \searrow 0$ for $x>0$, proving that $P_\varepsilon(x)\in(0,1]$ for $x>0$.
	The fact that $P_\varepsilon(x)\in[-1,0)$ for $x<0$ follows analogously.
	The limits of $P_\varepsilon$ as $x\to\pm\infty$ follow from the identity
	\begin{equation*}
		P_\varepsilon(x)=\frac{2x}{\sqrt{(x+1)^2+\varepsilon}+\sqrt{(x-1)^2+\varepsilon}}.
	\end{equation*}
	The monotonicity is an immediate consequence of the claim in \ref{enum:smoothProjDer_ii} proved below.
	
	\underline{Point \ref{enum:smoothProjDer_ii}.}
	Direct computations show that $P_\varepsilon'(x)\neq0 \ \forall x\in\R$ and $P'_\varepsilon(x)\xrightarrow{x\to\pm\infty}0$.
	Using the second derivative $P_\varepsilon''$, 
	we obtain that $P_\varepsilon'$ attains its unique maximum at $x=0$ with the value $P_\varepsilon'(0)=\frac{1}{\sqrt{1+\varepsilon}}$.
	Further, we have that $P'_\varepsilon(x)\to \frac{1}{2}(\mbox{sign}(x+1)-\mbox{sign}(x-1))$ pointwise as $\varepsilon\searrow 0$.
	
	\underline{Point \ref{enum:smoothProjDer_iii}.} Since the real-valued functions $P_\varepsilon$ and $P_\varepsilon'$ are bounded, the corresponding operators map into $L^\infty(\Omega)$. $P_\varepsilon\in C^\infty(\R)$ is globally Lipschitz,  since $P_\varepsilon'$ is bounded. 
	Hence the operator $P_\varepsilon \colon L^2(\Omega)\to L^2(\Omega)$ is globally Lipschitz because $\Omega$ is bounded.
	\underline{Point \ref{enum:smoothProjDer_iiii}.}
	Using the mean-value theorem we have for $x\in \R$
	
	\begin{equation}\label{eq: appendix scalar rate}
		|\proj_{[-1,1]}(x) - P_\varepsilon(x)|\leq \int_{0}^{\varepsilon}
		\bigg|\frac{\partial}{\partial \varepsilon}P_s(x) \bigg|ds
	\end{equation} 
	with the partial derivative
	\begin{equation*}
		\frac{\partial}{\partial \varepsilon} P_s(x)
		=
		\frac14 \bigg[
		\frac{1}{\sqrt{(x+1)^2+s}}
		-
		\frac{1}{\sqrt{(x-1)^2+s}}
		\bigg]
	\end{equation*}
	where we can estimate the right-hand side as
	\begin{equation}\label{eq: appendix scalar estimate rate}
		\bigg|
		\frac{\partial}{\partial \varepsilon} P_s(x)
		\bigg|
		\le
		\frac14 \bigg(
		\frac{1}{\sqrt{(x+1)^2 +s}}
		+
		\frac{1}{\sqrt{(x-1)^2 +s}}
		\bigg)
		=
		\frac{1}{2\sqrt s}
		.
	\end{equation}
	The result follows now by computing the integral in \eqref{eq: appendix scalar rate}. The estimate in \eqref{eq:nemistky_estimate} is a direct consequence of the previous estimate and the boundedness of the domain.
	
\end{proof}
\subsection{Solutions to the smoothed optimality systems}\label{sec:reg:opt}
Replacing $\proj _{[-1,1]}$ with $P_{\varepsilon}$ in \eqref{eq:smoothProj}, we obtain the smoothed optimality system
\begin{subequations}\label{eq:reg_OS}
	\begin{alignat}{3}
		A {y}+\varphi({y})&=f+{u} &&\text{ in }   H^{-1}(\Omega), \label{eq:reg_OS:state}\\
		\Big(A +\varphi'({y})\Big){p}&={y}-y_d &&\text{ in }   H^{-1}(\Omega), \label{eq:reg_OS:adjoint}\\
		{p}+\nu {u}+\mu P_{\varepsilon}\bigg(-\frac{{p}}{\mu}\bigg)&= 0&& {\text{ in } } L^2(\Omega),\label{eq:reg_OS:stationarity}
	\end{alignat}
\end{subequations}
for the triple $(u,y,p)\in L^2(\Omega)\times H_0^1(\Omega)\times H_0^1(\Omega)$ and its reduced form
\begin{equation}\label{eq:reg_OS_red}
	\begin{aligned}
		A {y}+\varphi({y})&=f-\frac{1}{\nu}\bigg({p}+\mu P_\varepsilon\bigg(-\frac{{p}}{\mu}\bigg)\bigg) &\text{ in }   H^{-1}(\Omega), \\
		\Big(A +\varphi'({y})\Big){p}&={y}-y_d &\text{ in }   H^{-1}(\Omega) \\
	\end{aligned}
\end{equation}
for ${y},{p}\in H_0^1(\Omega)$.
Note that, since the real function $p\in  \mathbb{R}\mapsto P_{\varepsilon}(-\frac{p}{\mu})$ is monotonically decreasing, the reduced system is a non-monotone, semi-linear system, i.e., we cannot prove the existence of solutions applying techniques from the theory of monotone operators.
Instead, we will construct a smooth optimal control problem, whose optimality system coincides with \eqref{eq:reg_OS}.
The auxiliary optimality system is obtained by replacing the non-smooth $L^1$-term in the cost function of our original problem with an appropriate differentiable.
This approach is nontrivial because the regularization operator is applied to the adjoint state $p$, while the cost functional of the regularized, auxiliary problem can only include terms in $u$ and $y$.

Thus, we begin with the ansatz
\begin{equation}
	\label{eq:smoothed_stationarity_condition}
	J_\varepsilon(u)=\frac{1}{2}\bigr\|S(u)-y_d\bigr\|_{L^2}^2+\frac{\nu}{2}\bigr\|u\bigr\|_{L^2}^2+\mu\regOp(u)
\end{equation}
with a differentiable functional $\regOp \colon L^2(\Omega)\to\mathbb{R}$.
Provided some assumptions on $\regOp$, which we will specify later,
one can find a necessary first-order optimality system for minimizers of the functional $J_\varepsilon$ over controls $u \in L^2(\Omega)$ using standard techniques.
The optimality system unsurprisingly consists of the lines \eqref{eq:reg_OS:state}-\eqref{eq:reg_OS:adjoint} and an ($L^2$-gradient) stationarity condition, which reads as
\begin{equation}\label{ref}
	{p}+\nu {u}+ \mu \nabla\regOp({u}) = 0
	\quad \text{in } L^2(\Omega)
	.
\end{equation}
Accordingly, the smoothed system \eqref{eq:reg_OS} and the optimality system of \eqref{eq:smoothed_stationarity_condition} coincide in the first two equations.
We now construct a $\regOp$ such that the relation
\begin{equation}
	\label{eq:smooth_counterpart_projection}
	\nabla\regOp({u})=P_{\varepsilon}\bigg(-\frac{{p}}{\mu}\bigg)
\end{equation}
is satisfied for $u$ and $p$ satisfying \eqref{ref} and for $\varepsilon>0$, establishing a clear connection between \eqref{ref} and the smoothed stationarity condition \eqref{eq:reg_OS:stationarity}.
Note that \eqref{eq:smooth_counterpart_projection} is a smooth counterpart to the explicit form of the Clarke-subderivative $\bar{\lambda}$ in \eqref{eq:lambda_ident}. 
Using \eqref{ref}, we obtain that \eqref{eq:smooth_counterpart_projection} holds if $\regOp$ satisfies the fixed-point condition
\begin{equation}\label{eq:fixed point _cond}
	\nabla\regOp({u})=P_{\varepsilon}\bigg(-\frac{{p}}{\mu}\bigg)=P_{\varepsilon}\bigg(\nabla\regOp({u})+\frac{\nu}{\mu}{u}
	\bigg).
\end{equation}

We begin showing the existence of such a functional $\regOp$ starting from \eqref{eq:fixed point _cond} by proving the existence of a scalar function $\regg_\varepsilon$ that satisfies the scalar counterpart to \eqref{eq:fixed point _cond}, that is
\begin{equation}\label{eq:defdeps}
	\regg_\varepsilon(x)=P_{\varepsilon}\bigg(\regg_\varepsilon(x)+\frac{\nu}{\mu} x\bigg)
\end{equation}
for every $x\in\R$ and for $\varepsilon>0$. Here, $\regg_\varepsilon$ plays the role of
$\nabla\regOp(\bar{u})$ in a pointwise sense.
Once $\regg_\varepsilon$ in $\mathbb{R}$ is obtained, $\regOp(\bar{u})$ is given
as the Nemytskii operator of the antiderivative of $\regg_\varepsilon$.

\begin{lemma}[Existence and properties of $\regg_\varepsilon$]\label{lem:fixPoiExistence}
	For every $\varepsilon>0$, there exists a unique, strictly monotonically increasing function $\regg_\varepsilon\in C^\infty(\mathbb{R}, [-1,1])$ satisfying \eqref{eq:defdeps} for all $x\in\mathbb{R}$.
	It has the following properties:
	\begin{enumerate}
		\item\label{eq:fixPoiExistence:1} $\lim\limits_{x\to\pm \infty} \regg_\varepsilon(x)=\pm 1$;
		\item\label{eq:fixPoiExistence:2} $\regg_\varepsilon(x)=0$ if and only if $x=0$;
		\item\label{eq:fixPoiExistence:3} $\regg_\varepsilon'(x)>0$ for all $x\in \R$;
		\item\label{eq:fixPoiExistence:4} $\regg_\varepsilon'$ is bounded and therefore $\regg_\varepsilon$ is globally Lipschitz continuous.
	\end{enumerate}
\end{lemma}

\begin{proof}
	For $x\in\mathbb{R}$, we define the  function $F_x\colon\mathbb{R}\to\mathbb{R}$ via $ F_x(\regg)\coloneqq P_{\varepsilon}(\regg+\frac{\nu}{\mu} x)$ and, using Point \ref{enum:smoothProjDer_ii} in Lemma \ref{lem:PProperties}, we obtain the bound
	\begin{equation*}
		|F'_x(\regg)|=|P_\varepsilon'(\regg+\frac{\nu}{\mu} x)|\leq\big\|P'_{\varepsilon}\big\|_{\infty} = \frac{1}{\sqrt{1+\varepsilon}}<1.
	\end{equation*}
	Therefore, the mean value theorem yields  that
	\begin{equation*}
		|F_x(\regg_1)-F_x(\regg_2)|\leq |F'_x(\bar{\regg})||\regg_1-\regg_2|\leq \frac{1}{\sqrt{1+\varepsilon}}|\regg_1-\regg_2|
	\end{equation*}
	for all $x,d_1,d_2\in\mathbb{R}$ and $\varepsilon>0$. Hence, the Banach fixed-point theorem yields a unique fixed point $\regg_\varepsilon(x)$ satisfying \eqref{eq:defdeps}, which defines the unique function $\regg_\varepsilon$.
	
	The boundedness of $\regg_\varepsilon(x) \in [-1,1]$ follows from the fixed-point equation \eqref{eq:defdeps} and the boundedness of $P_{\varepsilon}$.
	Using the boundedness of $\regg_\varepsilon$, we see that $\regg_\varepsilon(x)+\frac{\nu}{\mu} x\xrightarrow{x\to\pm\infty}\pm \infty$, which implies that $ \regg_\varepsilon(x)=P_{\varepsilon}\big(\regg_\varepsilon(x)+\frac{\nu}{\mu} x\big)\to\pm1$ as $x\to \pm \infty$ by Lemma \ref{lem:PProperties} Point \ref{enum:smoothProjDer_i}.
	
	To obtain the regularity of $\regg_\varepsilon$, we apply the implicit function theorem to the smooth function $I(x,\regg)\coloneqq \regg-P_{\varepsilon}\big(\regg+\frac{\nu}{\mu} x\big)$ for $(x,\regg)\in \mathbb{R}\times[-1,1]$.
	Indeed, by construction, we have that $I(x,\regg_\varepsilon(x))=0$ for all $x\in\mathbb{R}$.
	For the partial derivative of $I$ in $\regg$, we have that
	\begin{equation*}
		\partial_d I(x,\regg)=1-P_\varepsilon'\bigg( \regg+\frac{\nu}{\mu} x\bigg) >1-\frac{1}{\sqrt{1+\varepsilon}}>0
	\end{equation*}
	for all $x$, $\regg \in\R$.
	For every arbitrary $x_0\in \R$, we obtain a $\delta>0$ and a unique, smooth function $s\colon(x_0-\delta,x_0+\delta)\to[-1,1]$ with $I(x,s(x))=0$ for all  $x\in (x_0-\delta,x_0+\delta)$ from the implicit function theorem.
	Because of the uniqueness of $\regg_\varepsilon$, we have $\regg_\varepsilon=s$ on $(x_0-\delta,x_0+\delta)$. Since $x_0$ is arbitrarily chosen, we get $\regg_\varepsilon\in C^\infty(\mathbb{R},[-1,1])$.
	It remains to show points \ref{eq:fixPoiExistence:2}, \ref{eq:fixPoiExistence:3} and \ref{eq:fixPoiExistence:4} and the monotonicity of $\regg_\varepsilon$.
	Differentiation of \eqref{eq:defdeps} gives
	\begin{equation}\label{s'}
		\regg_\varepsilon'(x)=\frac{\nu}{\mu}\ P_\varepsilon'\bigg( \regg_\varepsilon(x)+\frac{\nu}{\mu} x\bigg)\biggl( 1- \ P_\varepsilon'\bigg( \regg_\varepsilon(x)+\frac{\nu}{\mu} x\bigg) \biggr)^{-1}.
	\end{equation}
	Note that $P_\varepsilon'$ maps into $(0,1)$ and is bounded away from $1$ for fixed $\varepsilon$ by Point \ref{enum:smoothProjDer_ii} in Lemma \ref{lem:PProperties}, this immediately implies the boundedness and the positivity of $\regg_\varepsilon'$ and hence the strict monotonicity of $\regg_\varepsilon$.
	This implies that $\regg_\varepsilon$ has exactly one root, which has to be at $x=0$ because $\regg_\varepsilon(x)=0$ implies $P_\varepsilon(\frac{\nu}{\mu}x)=0$, which is exactly the case when $x=0$.
\end{proof}

\begin{lemma}[The antiderivative $\reg_\varepsilon$]\label{eq:antider_S}
	Let $\reg_\varepsilon(x)\coloneqq\int\limits_0^x \regg_\varepsilon(s) \mathrm d s$ for $x\in\mathbb{R}$. Then $\reg_\varepsilon$ is bounded from below by $0$, strictly convex, non-expansive and $\reg_\varepsilon(0)=0$.
\end{lemma}

\begin{proof}
	By definition of $\reg_\varepsilon$ and Lemma \ref{lem:fixPoiExistence}, we have $\reg_\varepsilon''(x)=\regg_\varepsilon'(x)>0$, hence $\reg_\varepsilon$ is strictly convex.
	Further, the monotonicity of $\regg_\varepsilon$ and \eqref{eq:regOpDef} in Lemma \ref{lem:fixPoiExistence} imply that $\reg_\varepsilon'(x)=\regg_\varepsilon(x) <0$ ($=0$ or $>0$) if and only if $x<0$ ($x=0$ or $x>0$). Thus $\reg_\varepsilon$ is strictly decreasing for $x<0$ and strictly increasing for $x>0$. Since $\reg_\varepsilon$ is continuous with $\reg_\varepsilon(0)=0$ by construction, $\reg_\varepsilon$ is bounded from below by $0$.
	Global Lipschitz continuity with constant $L=1$ is a direct consequence of the boundedness of $\regg_\varepsilon$.
\end{proof}

Now, the goal is to extend the constructed real functions $\regg_\varepsilon$
and $\reg_\varepsilon$ into functional operators on the space $L^2(\Omega)$.
To this end, we define the operator $\regOp\colon L^2(\Omega)\to \mathbb{R}$ as
\begin{equation}\label{eq:regOpDef}
	\regOp(u)\coloneqq \int_\Omega \reg_\varepsilon(u(x))\ \mathrm  d x,
\end{equation}
whose properties are studied in the following lemma.
\begin{lemma}[Properties of $\regOp$]\label{lemma S Nem }
	The operator $\regOp$ is strictly convex, weakly lower semi-continuous, bounded from below, globally Lipschitz continuous and continuously Fr\'echet differentiable at $u\in L^2(\Omega)$ with
	\begin{equation}\label{eq:regOpDerDef}
		\regOp'(u)h=\int_\Omega   \reg_\varepsilon'(u(x))h(x)\ \mathrm dx =\int_\Omega   \regg_\varepsilon(u(x))h(x)\ \mathrm dx\quad \forall h\in L^2(\Omega).
	\end{equation}
	Moreover, $\regOp(0)=0$ and $\reg_\varepsilon'\colon L^2(\Omega)\to L^2(\Omega)'$ is Lipschitz continuous.
\end{lemma}

\begin{proof}
	Convexity and boundedness of $\regOp$ follow from the properties of $ \reg_\varepsilon$ proven in Lemma \ref{eq:antider_S}.
	Lipschitz continuity is a consequence of the Lipschitz continuity of $\reg_\varepsilon$ and the continuous embedding of $L^2(\Omega)$ into $L^1(\Omega)$.
	Similarly, $\regOp(0)=0$ is immediately clear from Point \ref{eq:fixPoiExistence:2} in Lemma \ref{lem:fixPoiExistence}.
	
	Let us now focus on differentiability and semi-continuity.
	We begin with the differentiability.
	First, the boundedness of $ \reg_\varepsilon'= \regg_\varepsilon$ by $\pm1$ from Lemma \ref{lem:fixPoiExistence} implies that
	\begin{equation}\label{eq:S_bounded}
		\| \reg_\varepsilon'\circ u\|_{L^2}\leq |\Omega| <\infty
	\end{equation}
	for all $u\in L^2(\Omega)$, i.e., the operator is well defined.
	Now, for $u,h\in L^2(\Omega)$, we have
	\begin{equation}\label{eq:S_Frechet}
		\medmuskip=-1mu
		\thinmuskip=-1mu
		\thickmuskip=-1mu
		\nulldelimiterspace=0.9pt
		\scriptspace=0.9pt
		\arraycolsep0.9em
		\bigg|\regOp(u+h)-\regOp(u)-\regOp'(u)h \bigg|=\bigg|\int_\Omega { \reg}_\varepsilon\big(u(x)+h(x)\big)-{ \reg}_\varepsilon\big(u(x)\big)-{ \reg'_\varepsilon}\big(u(x)\big)h\ \mathrm  dx\bigg|.
	\end{equation}
	For almost every $x\in\Omega$, the mean value theorem yields $\delta(x)\in (0,1)$ such that
	\begin{equation}
		{ \reg_\varepsilon}\big(u(x)+h(x)\big)-{ \reg_\varepsilon}\big(u(x)\big)= \reg_\varepsilon'\big(u(x)+\delta(x)h(x)\big)h(x).\nonumber
	\end{equation}
	Inserting this in \eqref{eq:S_Frechet}
	and denoting by $L_{\regg_\varepsilon}$ the Lipschitz constant of $\regg_{\varepsilon}$,
	we obtain that
	\begin{align*}
		\bigg|\regOp(u+h)-\regOp(u)-\regOp'(u)h \bigg|&=\bigg|\int_\Omega \bigg( \reg_\varepsilon'\big(u(x)+\delta(x)h(x)\big)-{ \reg}_\varepsilon'\big(u(x)\big)\bigg)h(x) \ \mathrm dx\bigg|\nonumber\\
		&\leq\int_\Omega \left| \regg_\varepsilon\big(u(x)+\delta(x)h(x)\big)-{ \regg}_\varepsilon\big(u(x)\big)\right|\left|h(x)\right| \ \mathrm dx\nonumber\\
		&\leq L_{\regg_\varepsilon}\int_\Omega \left|\delta(x)\right|\left|h(x)\right|^2 \ \mathrm dx
		\leq  L_{\regg_\varepsilon} \|h\|_{L^2(\Omega)}^2, \nonumber
	\end{align*}
	showing the differentiability result.
	
	Further, we show that the map $\regOp'\colon L^2(\Omega)\to L^2(\Omega)'$ is Lipschitz continuous.
	By Lemma \ref{lem:fixPoiExistence}, $\regg_\varepsilon = \reg_\varepsilon'$ is globally Lipschitz. Thus
	\begin{align*}
		\|\regOp'(u)-\regOp'(v)\|_{L^2(\Omega)'}&=\sup\limits_{\|z\|_{L^2}=1}\bigg|\int_\Omega \big( \reg_\varepsilon'\big(u(x)\big)- \reg_\varepsilon'\big(v(x)\big)\big) z \ dx\bigg|\\
		&\leq  \| \reg_\varepsilon'\circ u- \reg_\varepsilon'\circ v\big\|_{L^2(\Omega)}\\
		&\leq  L_{\regg_\varepsilon} \|u-v\|_{L^2(\Omega)}
	\end{align*}
	for $u,v\in L^2(\Omega)$.
	Hence, $\regOp$ is continuously differentiable.
	The weak lower semi-continuity follows from the fact that $\regOp$ is convex and continuous.
\end{proof}
We can finally prove the existence of a solution to the smoothed optimality system.
\begin{theorem}[Solvability of the smoothed optimality system \eqref{eq:reg_OS}]\label{Theorem: Solv Reg} For all $\varepsilon>0$, the auxiliary optimal control problem
	\begin{equation}\label{eq:P_eps}
		\min\limits_{u\in L^2(\Omega)}
		J_\varepsilon(u)=\frac{1}{2}\bigr\|S(u)-y_d\bigr\|_{L^2}^2+\frac{\nu}{2}\bigr\|u\bigr\|_{L^2}^2+\mu \regOp(u) 
	\end{equation}
	admits a global solution $u\in L^2(\Omega)$ and $y,p\in H_0^1(\Omega)\cap L^\infty(\Omega)$ that is a solution to the smoothed optimality system \eqref{eq:reg_OS}.
\end{theorem}

\begin{proof}
	Using the properties of $\regOp$ from Lemma \ref{lemma S Nem }, we see that \eqref{eq:P_eps} is well-posed and admits a solution, cf.~Lemma \ref{lem:minEx} and, again, \cite[Sec. 4.4.2]{Troeltzsch05}.
	Necessarily, its first-order optimality system admits a solution $u\in L^2(\Omega)$ and $y,p\in H^1(\Omega)\cap L^\infty(\Omega)$, cf.~also Lemmas \ref{lem:forwardSol} and \ref{lem:adjointSol}.
	This system is given by the state and adjoint equation and the optimality condition
	\begin{equation}\label{opt condition}
		p+\nu u+ \mu \nabla\regOp(u)=0 \text{ in } L^2(\Omega).
	\end{equation}
	With $\regOp'(u)$ given by \eqref{eq:regOpDerDef}, we have that $\nabla\regOp(u)= \reg'_\varepsilon \circ u\in L^2(\Omega)$ and by construction of $ \reg'_\varepsilon=\regg_\varepsilon$ via \eqref{eq:defdeps}, \eqref{opt condition} implies
	\begin{equation}
		{ \reg}'_\varepsilon(u)=P_{\varepsilon}\bigg({ \reg}'_\varepsilon(u)+\frac{\nu}{\mu} u\bigg)=P_{\varepsilon}\bigg(-\frac{p}{\mu}\bigg) \in L^2(\Omega).
	\end{equation}
	Therefore the solvable optimality system of \eqref{eq:P_eps} coincides with \eqref{eq:reg_OS}.
\end{proof}

\begin{remark}
	Due to the fixed-point approach of defining $\regg_\varepsilon$, we do not obtain an explicit representation for $\reg_\varepsilon$.
	Numerically, one can observe that $\reg_\varepsilon$ behaves like a smoothing of the absolute value function.
\end{remark}

\subsection{Convergence analysis}\label{sec:convAna}
In this section, we study the behavior of solutions to the regularized system \eqref{eq:reg_OS}
as the smoothing parameter $\varepsilon$ tends to $0$.
In particular, we prove that these converge weakly to a solution to the original non-smooth system \eqref{eq:non-smooth_OS_lam} in the sense that weak accumulation points of sequences of the regularized solutions are solutions of the non-smooth optimality system  \eqref{eq:non-smooth_OS_lam}.
\begin{theorem}
	\label{thm:convResult}
	Let $(y_n)_{n\in \N},(p_n)_{n\in \N}\subseteq H^1_0(\Omega)\cap L^{\infty}(\Omega)$ be sequences of solutions of \eqref{eq:reg_OS_red} for $\varepsilon=\varepsilon_n\searrow 0$ as $n \rightarrow \infty$
	and let $y$ and $p$ be weak $H^1(\Omega)$-accumulation points of these sequences, respectively.
	Then $y$ and $p$ are solutions of \eqref{eq:non-smooth_OS_reduced}, and any subsequences of $(y_n)_{n\in \N},(p_n)_{n\in \N}$ converging weakly to $y$ and $p$, respectively, converge strongly in $H_0^1(\Omega)\cap L^{\infty}(\Omega)$.
\end{theorem}
\begin{proof}
	We extract weakly convergent subsequences (which we tacitly denote with the same symbols as the original sequences) that satisfy \eqref{eq:reg_OS_red} for $\varepsilon=\varepsilon_n$, i.e.,
	\begin{equation}\label{eq:OS_regularized}
		\begin{aligned}
			A y_n+\varphi(y_n)&=f-\frac{1}{\nu}\bigg(p_n+\mu P_{\varepsilon_n}\bigg({-\frac{p_n}{\mu}}\bigg)\bigg) &&\text{ in } H^{-1}(\Omega), \\
			\Big(A +\varphi'(y_n)\Big)p_n&=y_n-y_d &&\text{ in }  H^{-1}(\Omega),
		\end{aligned}
	\end{equation}
	and such that $y_n \rightharpoonup y$ and $p_n \rightharpoonup p$ in $H_0^1(\Omega)$.
	By Rellich's compact embedding theorem, we obtain that $y_n\to y$ and $p_n\to p$ strongly in $L^2(\Omega)$.
	This especially implies that $-\frac{1}{\nu}\big(p_n+\mu P_{\varepsilon_n}({-\frac{p_n}{\mu}})\big)$ is uniformly bounded in $L^2(\Omega)$. Thus, using the a-priori estimate in Lemma \ref{lem:forwardSol}, we obtain that there exists an $M_p >0$ such that
	\begin{equation*}
		\|y_n\|_{H^1}+\|y_n\|_{L^{\infty}}\leq C \bigg\|f-\frac{1}{\nu}\bigg(p_n+\mu P_{\varepsilon_n}\biggr(-\frac{p_n}{\mu}\biggr)\bigg)-\varphi(0)\bigg\|_{L^2}\leq M_y.
	\end{equation*}
	The corresponding a-priori estimate for the adjoint equation in Lemma \ref{lem:adjointSol} yields
	$\|p_n\|_{H^1}+\|p_n\|_{L^{\infty}}\leq C \|y_n-y_d\|_{L^2}\leq M_p$,
	for an $M_p>0$.
	Defining $M\coloneqq\max(M_y,M_p)$, the set $\{ v\in H_0^1(\Omega)\cap L^\infty(\Omega): \|v\|_{L^{\infty}}\leq M\}$ is convex and closed in the $H^1(\Omega)$-topology and therefore weakly closed in $H^1(\Omega)$.
	Hence, we obtain that
	\begin{equation}\label{boundedness}
		\|y_n\|_{L^{\infty}},\|p_n\|_{L^{\infty}},\|y\|_{L^{\infty}},\|p\|_{L^{\infty}}\leq M \qquad \forall n\in \N.
	\end{equation}
	By \cite[Lemma 4.11]{Troeltzsch05}, the Nemytskii operator $\varphi\colon L^\infty(\Omega)\to L^\infty(\Omega)$ associated with $\varphi$ satisfies
	\begin{equation*}
		\|\varphi(y_n)-\varphi({y})\|_{L^2}\leq K_\varphi(M)\|y_n-y\|_{L^2}{\xrightarrow{n\to\infty}}0
	\end{equation*}
	for a $K_\varphi(M)>0$, i.e., $\varphi(y_n)\to\varphi(y)$ in $L^2(\Omega)$.
	Since $\psi((y,p)\in\mathbb{R}^2) \coloneqq \varphi'(y)p\in \mathbb{R}$ is continuously differentiable and therefore locally Lipschitz, a direct modification to \cite[Lemma 4.11]{Troeltzsch05} to account for functions $\psi\colon \R^2\to \R$ yields that the associated Nemytskii operator $\psi\colon L^{\infty}(\Omega)^2\to L^{\infty}(\Omega) $ satisfies
	\begin{equation*}
		\|\psi(y_n,p_n)-\psi(y,p)\|_{L^2}\leq K_\psi(M)\|(y_n,p_n)-(y,p)\|_{L^2} \to 0,
	\end{equation*}
	for a $K_\psi(M)>0$, i.e., $\psi(y_n,p_n)\to\psi(y,p)$ in $L^2(\Omega)$.
	
	Since $A$ is linear and bounded from $H_0^1(\Omega)$ to $H^{-1}(\Omega)$,
	we know that $A$ is weakly continuous, which implies that $Ay_n\rightharpoonup Ay$ and $Ap_n\rightharpoonup Ap$ in $H^{-1}(\Omega)$.
	Finally, using the Lipschitz property of $P_{\varepsilon_n}$ in Lemma \ref{lem:PProperties}, we obtain that
	\begin{align*}
		\bigg\|P_{\varepsilon_n}&\bigg(-\frac{p_n}{\mu}\bigg)-\mbox{proj}_{[-1,1]}\bigg(-\frac{p}{\mu}\bigg)\bigg\|_{L^2}\\
		\leq &  \bigg\|P_{\varepsilon_n}\bigg(-\frac{p_n}{\mu}\bigg)-P_{\varepsilon_n}\bigg(-\frac{p}{\mu}\bigg)\bigg\|_{L^2}+\bigg\|P_{\varepsilon_n}\bigg(-\frac{p}{\mu}\bigg)-\mbox{proj}_{[-1,1]}\bigg(-\frac{p}{\mu}\bigg)\bigg\|_{L^2}\\
		\leq& L\bigg\|p_n-p\bigg\|_{L^2}+\bigg\|P_{\varepsilon_n}\bigg(-\frac{p}{\mu}\bigg)-\mbox{proj}_{[-1,1]}\bigg(-\frac{p}{\mu}\bigg)\bigg\|_{L^2}\to 0,
	\end{align*}
	where we used dominated convergence for the second term, since by Lemma \ref{lem:PProperties} we have the pointwise convergence, and the integrand is bounded.
	Summarizing, we have
	\begin{equation}\label{eq:summary}
		\begin{aligned}
			&Ay_n\rightharpoonup Ay \; \mbox{ and } \; Ap_n\rightharpoonup Ap &&\mbox{ in } H^{-1}(\Omega),\\
			&y_n\to y \; \mbox{ and } \; p_n\to p &&\mbox{ in } L^2(\Omega),\\
			&P_{\varepsilon_n}\bigg(-\frac{p_n}{\mu}\bigg)\to\mbox{proj}_{[-1,1]}\bigg(-\frac{p}{\mu}\bigg)  &&\mbox{ in } L^2(\Omega),\\
			&\varphi(y_n)\to \varphi(y) \mbox{ and } \psi(y_n,p_n)\to \psi(y,p) &&\mbox{ in } L^2(\Omega).
		\end{aligned}
	\end{equation}
	Using \eqref{eq:summary}, we can take the limit in \eqref{eq:OS_regularized} and obtain that $y,p$ are solutions of (\ref{eq:reg_OS_red}).
	With the Lipschitz property of the solution operator $S$ in Lemma \ref{lem:forwardSol}, we obtain strong convergence in $H_0^1(\Omega)\cap L^\infty(\Omega)$, because
	\begin{equation*}
		\medmuskip=-0.0mu
		\thinmuskip=-0.0mu
		\thickmuskip=-0.0mu
		\nulldelimiterspace=1.3pt
		\scriptspace=1.3pt
		\arraycolsep1.3em
		\big\| y-y_n\big\|_{H^1}+\big\| y-y_n\big\|_{L^\infty}\leq \frac{L}{\nu} \bigg\| p-p_n+\mu \bigg(\mbox{proj}_{[-1,1]}\bigg(-\frac{p}{\mu}\bigg)- P_{\varepsilon_n}\bigg(-\frac{p_n}{\mu}\bigg)\bigg) \bigg\|_{L^2}\to 0.
	\end{equation*}
	Similarly, the strong convergence of $p_n\to p$ in $H^1(\Omega)\cap L^\infty(\Omega)$ follows from the corresponding Lipschitz condition in Lemma \ref{lem:adjointSol}.
\end{proof}

\begin{remark}[$L^\infty(\Omega)$-convergence of $(u_n)_{n\in \N}$]
	Because $p_n\to p$ in $L^\infty(\Omega)$ and $u_n = -\frac{1}{\nu}({p}_n+\mu P_{\varepsilon_n}(-\frac{{p}_n}{\mu}))$, the continuity of $P_{\varepsilon_n}\colon L^\infty(\Omega) \rightarrow L^\infty(\Omega)$ even yields convergence $u_n\to u$ in $L^\infty(\Omega)$.
\end{remark}

The following technical lemma guarantees that there in fact exist sequences of solutions to the smoothed optimality systems that possess accumulation points.
\begin{lemma}[Existence of weak accumulation points]
	\label{existence of weak cumulation}
	Let $\varepsilon_n\searrow 0$.
	Then there exist sequences $(u_n)_{n\in \N}\subseteq L^2(\Omega)$ and $(y_n)_{n\in \N},(p_n)_{n\in \N}\subseteq H^1(\Omega)\cap L^\infty(\Omega)$ of global solutions to the auxiliary optimization problems \eqref{eq:P_eps} (and its necessary optimality system \eqref{eq:reg_OS}) for each $\varepsilon=\varepsilon_n$, such that $(u_n)_{n\in \N}$ has a weak $L^2$-accumulation point $u$ and $(y_n)_{n\in \N},(p_n)_{n\in \N}$ have weak $H^1$-accumulation points $y,p\in H^1(\Omega)\cap L^\infty(\Omega)$.
\end{lemma}

\begin{proof}
	The sequences exist because of Theorem \ref{Theorem: Solv Reg}, and we have that $y_n=S(u_n)$. Since $u_n$ is a global minimizer for $J_{\varepsilon_n}$ (thus $J_{\varepsilon_n}(u_n)\leq J_{\varepsilon_n}(0)$) and using the identity $\mathcal{\reg}_{\varepsilon_n}(0)=0$ from Lemma \ref{lemma S Nem }, we obtain
	\begin{equation} \label{Global_estimate}
		\frac{\nu}{2}\bigr\| u_n\bigr\|_{L^2}^2 \leq {J}_{{\varepsilon_n}}(u_n)\leq {J}_{{\varepsilon_n}}(0)= J(S(0),0)= \frac{1}{2}\bigr\| S(0)-y_d\bigr\|^2_{L^2}<\infty.
	\end{equation}
	Thus, $(u_n)_{n\in \N}$ is uniformly bounded in $L^2(\Omega)$. Using the a-priori estimates for $y_n$ from Lemma \ref{lem:forwardSol}, we get
	$\bigr\|y_n\bigr\|_{H^1}+\bigr\|y_n\bigr\|_{L^{\infty}}\leq C \bigr\|f+u_n-\varphi(0)\bigr\|_{L^2}$.
	Hence, $(y_n)_{n\in \N}$ is uniformly bounded in $H_0^1(\Omega)$.
	By the corresponding a-priori estimate of Lemma \ref{lem:adjointSol}, we get that $(p_n)_{n\in \N}$ is uniformly bounded in $H_0^1(\Omega)$.
	The existence of weak accumulation points follows from the reflexivity of Hilbert spaces.
\end{proof}
Now, we are interested in the question whether or not the convergence rate of order $\sqrt{\varepsilon}$ from Lemma \ref{lem:PProperties} (see~\eqref{eq:nemistky_estimate}) carries over to the convergence of the solutions to the regularized optimality systems, cf.~Theorem \ref{thm:convResult}.
Consider the product space $V=H^1_0(\Omega)^2$, its dual $V'=H^{-1}(\Omega)^2$ and let $x=(y,p)\in V$.
We introduce the ($\varepsilon$-regularized) map $F: V\times[0,\infty) \to V'$ as
\begin{equation*}
	{F}(x, \varepsilon) \coloneqq
	\begin{bmatrix}
		Ay + \varphi({y})-f+\frac{1}{\nu}\big({p}+\mu P_\varepsilon(-\frac{{p}}{\mu})\big) \\
		Ap + \varphi'({y}){p}-{y}+y_d \\
	\end{bmatrix}.
\end{equation*}
Thus, the optimality system \eqref{eq:reg_OS_red} reads as
\begin{equation}\label{zeropoint}
	{F}( x, \varepsilon)=0 \quad \mbox{ in } V'.
\end{equation}
Well-posedness is guaranteed by Theorem~\ref{Theorem: Solv Reg} and we denote the solution for $\varepsilon>0$ by $x(\varepsilon)=(y(\varepsilon), p(\varepsilon))$. The partial derivatives of $F$ are
\begin{equation*}
	\begin{aligned}
		\frac{\partial}{\partial x}F(x,\varepsilon)
		&=
		\begin{bmatrix}
			A+\varphi'(y) & \frac{1}{\nu}(\id - P_\varepsilon'(-\frac{p}{\mu}))\\
			\varphi''(y)p - \id & A + \varphi'(y)
		\end{bmatrix} \in \mathcal L(V, V'),
		\\
		\frac{\partial}{\partial \varepsilon}F(x,\varepsilon)
		&=
		\begin{bmatrix}
			\frac{\mu}{\nu} \frac{\partial}{\partial \varepsilon} P_\varepsilon(-\frac{p}{\nu})\\
			0
		\end{bmatrix} \in \mathcal L([0, \infty), V')
		.
	\end{aligned}
\end{equation*}
We have the following sufficient condition for the convergence rate of the regularized solutions.
\begin{theorem}
	\label{thm:convergencerate}
	Assume that $\frac{\partial}{\partial \varepsilon} F(x(\varepsilon), \varepsilon)$ is continuously invertible for all $\varepsilon >0$ and its inverse is bounded independently of $\varepsilon$.
	Then, for $x(\varepsilon)$ and its $V$-limit $x$, we have the asymptotic
	\begin{equation*}
		\|x-x(\varepsilon)\|_{{H_0^1(\Omega)}^2} = \mathcal{O}(\sqrt{\varepsilon}) \quad \mbox{ as } \varepsilon\to 0.
	\end{equation*}
\end{theorem}
\begin{proof}
	Given the assumption, we can apply the implicit function theorem to $F(x(\varepsilon),\varepsilon)=0$ and write
	\begin{equation*}
		0 
		= 
		\frac{d}{d\varepsilon} F(x(\varepsilon),\varepsilon) 
		= 
		\frac{\partial}{\partial x} F(x,\varepsilon)	\dot x(\varepsilon) + \frac{\partial}{\partial \varepsilon}F(x(\varepsilon),\varepsilon)
	\end{equation*}
	and therefore
	\begin{equation}
		\dot x(\varepsilon) = -\frac{\partial}{\partial x}F(x(\varepsilon),\varepsilon)^{-1}F_\varepsilon(x(\varepsilon),\varepsilon).
	\end{equation}
	Hence, $\varepsilon \mapsto x(\varepsilon)$ is continuously differentiable, and we can apply the fundamental theorem of calculus to obtain
	\begin{equation*}
		\|x(\varepsilon) - x\|_{V}
		\le
		\int_{0}^{\varepsilon}
		\|\dot x(s)\|_{\mathcal L(\R,V)} ds,
	\end{equation*}
	Since $\frac{\partial}{\partial \varepsilon}F(x(\varepsilon), \varepsilon)^{-1}$ is assumed to be bounded independently of $\varepsilon$, we have
	\begin{equation*}
		\begin{aligned}
			\Big\| \dot x(s) \Big\|_{\mathcal L(\R,V)}
			&= \Big\|\frac{\partial}{\partial x} F (x(s),s)^{-1}\frac{\partial}{\partial \varepsilon} F (x(s),s) \Big\|_{\mathcal L(\R,V)}	\\
			&\le  C \Big\|\frac{\partial}{\partial \varepsilon} F (x(s),s) \Big\|_{\mathcal L(\R,V')}
			= C  \Big\|\frac{\partial}{\partial \varepsilon} F (x(s),s) \Big\|_{ V'} \\
			& = C \tfrac{\mu}{\nu}\Big\|\tfrac{\partial}{\partial\varepsilon} P_s(-\tfrac{p}{\nu})  \Big\|_{ H^{-1}(\Omega)}
			= C \tfrac{\mu}{\nu} \sup\limits_{\|v\|_{L^2(\Omega)}=1}|\langle \tfrac{\partial}{\partial\varepsilon} P_s(-\tfrac{p}{\nu}), v \rangle_{L^2(\Omega)}|\\
			& \le C \tfrac{\mu}{\nu} \Big\|  \tfrac{\partial}{\partial\varepsilon} P_s(-\tfrac{p}{\nu}) \Big\|_{L^2(\Omega)}
			\le  \tfrac{C  |\Omega| \mu}{2\nu}\tfrac{1}{\sqrt{s}},
		\end{aligned}
	\end{equation*}
	where we used the estimate \eqref{eq: appendix scalar estimate rate} in the last inequality. Now the claim follows by integrating the estimate.
\end{proof}
\begin{corollary}\label{cor:convergence}
	Under the additional assumption that $\nu$ is sufficiently large and that $x(\varepsilon)$ is a global solution for $\varepsilon>0$, we have
	\begin{equation*}
		\|x-x(\varepsilon)\|_{{H_0^1(\Omega)}^2} = \mathcal{O}(\sqrt{\varepsilon}) \quad \mbox{ as } \varepsilon\to 0.
	\end{equation*}
\end{corollary}
\begin{proof}
	We show that for $\nu$ large enough, the derivative $	\frac{\partial}{\partial x} F (x(\varepsilon),\varepsilon)$ is continuously invertible and its inverse is bounded independently of $\varepsilon$. In order to do that, we consider the Schur complement of $\frac{\partial}{\partial x} F (x(\varepsilon),\varepsilon)$ given by
	\begin{equation}\label{eq:condition_invert}
		S = (A+\varphi'(y)) + \tfrac{1}{\nu}(\id - P_\varepsilon'(-\tfrac{p}{\mu})(A + \varphi'(y))^{-1}(	\varphi''(y)p - \id)
	\end{equation}
	Since the sequence $x(\varepsilon)$ of global minimizers is uniformly bounded in $V$ w.r.t. $\nu$ and $\varepsilon$ (by similar arguments as in Lemma \ref{existence of weak cumulation}), there exists $\nu>0$ that ensures 
	\begin{equation}
		\big\|\tfrac{1}{\nu}(\id - P_\varepsilon'(-\tfrac{p}{\mu})(A + \varphi'(y))^{-1}(	\varphi''(y)p - \id)\big\| \leq q \|(A+\varphi'(y))^{-1} \|^{-1}
	\end{equation}
	for some $q\in(0,1)$. Using convergence of the corresponding Neumann series, this implies that $S$ is invertible with $\|S^{-1}\|\leq \tfrac{1}{1-q}\| (A+\varphi'(y))^{-1}\|$ bounded independently of $\varepsilon$.
\end{proof}
In Figure \ref{fig: convergence rates}, we numerically illustrate the convergence of the regularized solutions $x(\varepsilon)$ to the solution $x$ of the nonsmooth optimality system \eqref{eq:non-smooth_OS_lam}. 
In order to see the rate $ \mathcal{O}(\sqrt{\varepsilon})$ in a numerical experiment, the non-smoothness has to be active on gridpoints, otherwise faster rates of convergence could be observed.
For this reason, we consider a test configuration with an optimal control function $u$ that is zero inside the domain $\tilde \Omega=[0.25,0.75]^2\subset \Omega=[0,1]^2$ (see the right plot in Figure \ref{fig: convergence rates}). 
In particular, the adjoint state for our example is $p(x_1,x_2)= \mu s(x)s(y)$, where $s(x)=\chi_{\tilde \Omega} \max\{1,2 |\sin(2\pi x)|\}+\chi_{\Omega\setminus\tilde \Omega} 2 |\sin(2\pi x)| $ for $x\in [0,1]$. 
The corresponding desired state $y_d$ is given by \eqref{eq:ydpppp} using the corresponding solution of \eqref{eq:pppp}.
\begin{figure}[t]
	\centering
	\includegraphics[width=0.47\textwidth]{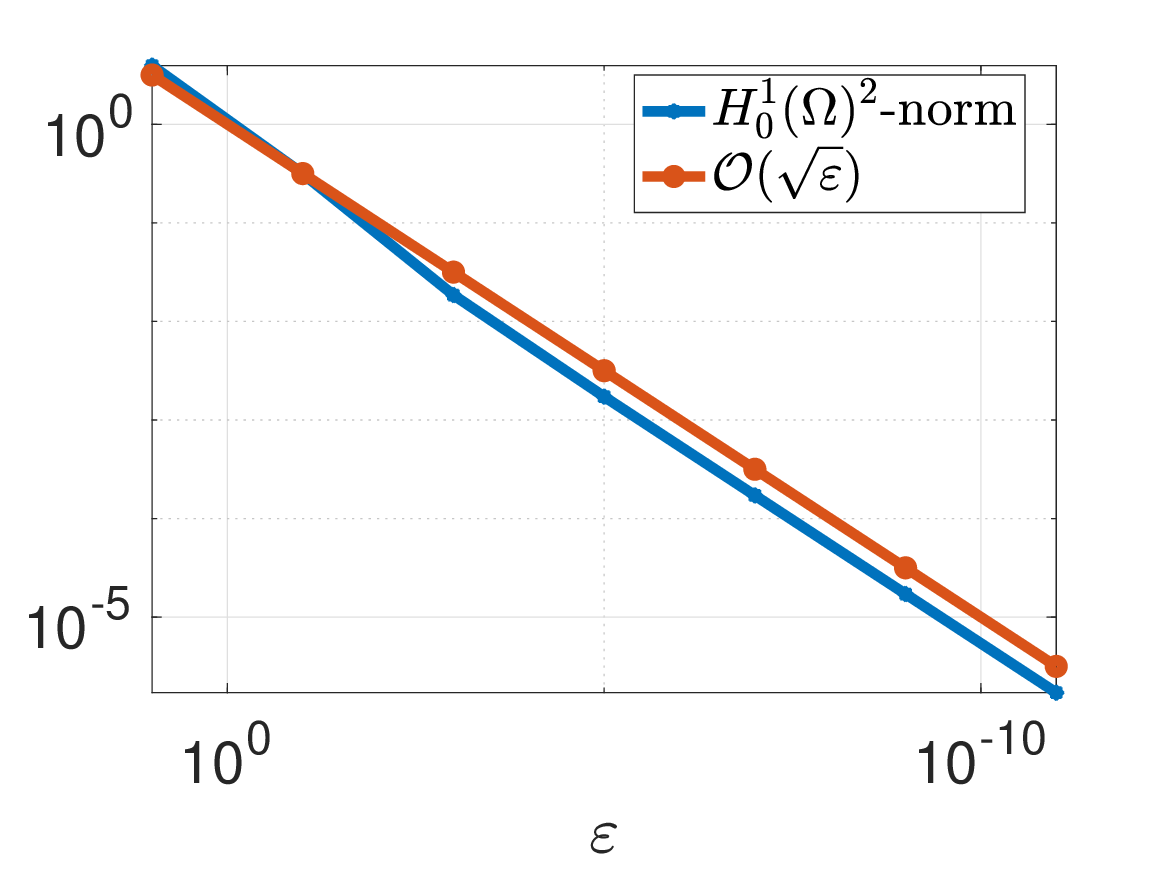}
	\includegraphics[width=0.47\textwidth]{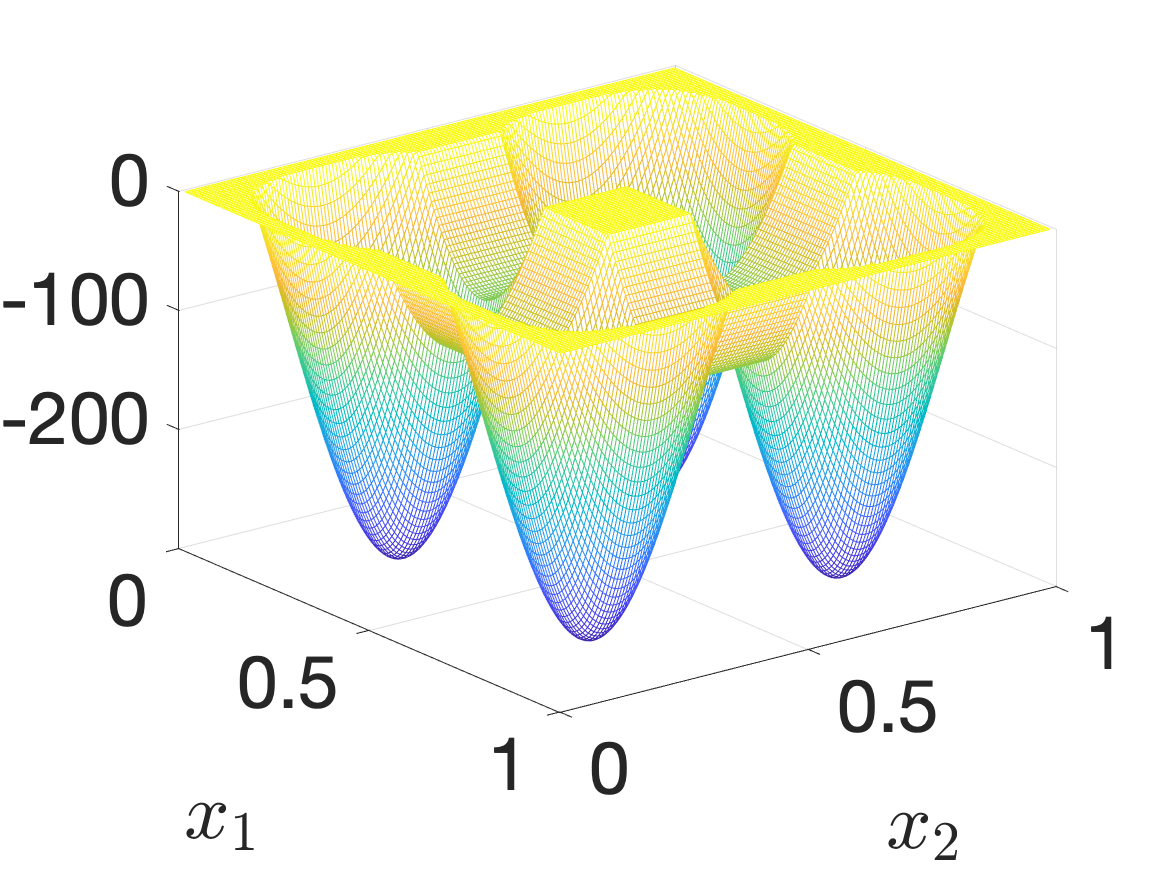}
	\caption[a]{Left: the $H^1_0(\Omega)^2$-norm of the difference $x(\varepsilon)-x$. Right: the corresponding optimal control.}\label{fig: convergence rates}
\end{figure}
Note that the interplay of the two regularization parameters $\nu$ and $\mu$ ultimately determines the sparsity pattern of the optimal control, where increasing $\nu$ yields a more distributed support of a smeared optimizer and increasing $\mu$ decreasing its support with less diffusive behavior in the solution.
In the smoothed case, the regularization parameter additionally influences the sparsity pattern of the solutions to the regularized optimality system, as one can tell from the smoothed stationarity condition in the last line of \eqref{eq:OS_regularized}.
However, for decreasing the smoothing parameter, the sparsity structure of the limiting solution to the nonsmooth system is typically recovered.
See Figure~\ref{fig: sparsity test} for the resulting optimal controls for the test configuration corresponding to \cite[Example 1]{Stadler:2009:1} for different values of the $L^1$-penalization parameter $\mu$ and different values of the regularization parameter $\varepsilon$. 
When $\varepsilon$ is rather large ($\varepsilon=1$, first row) the sparsity structure of the limiting control is essentially lost. 
However, for the chosen smaller values the correct sparsity is immediately recovered.
For $\varepsilon=10^{-2}$ (second row), the boundary of the sparsity region (support of the optimal control functions) is not yet sharp, but for $\varepsilon=10^{-3}$ and smaller regularization, the correct structure (third row) is obtained and, visually, the results can not be distinguished from the results corresponding to very small values of $\varepsilon$, like $\varepsilon=10^{-11}$.

\begin{figure}[t]
	\centering
	
	\includegraphics[width=0.3\textwidth]{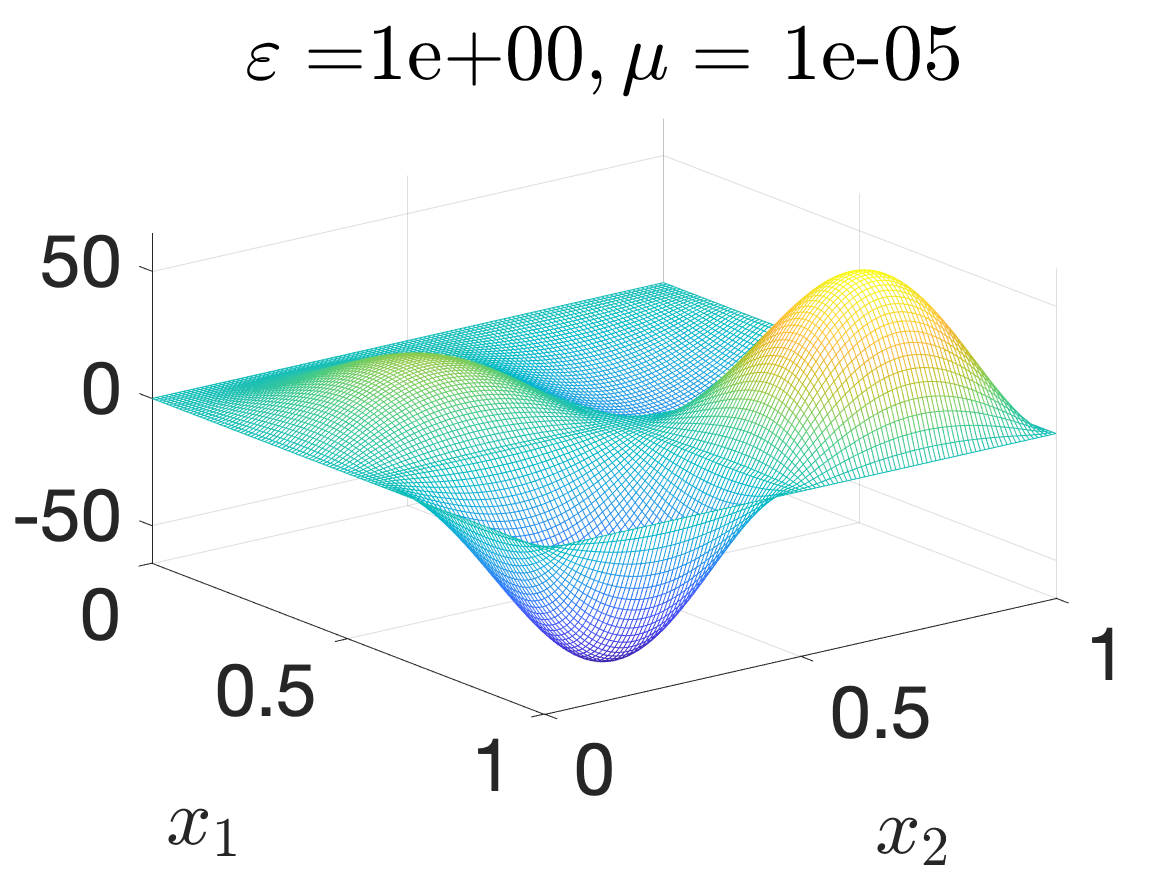}\hfil
	\includegraphics[width=0.3\textwidth]{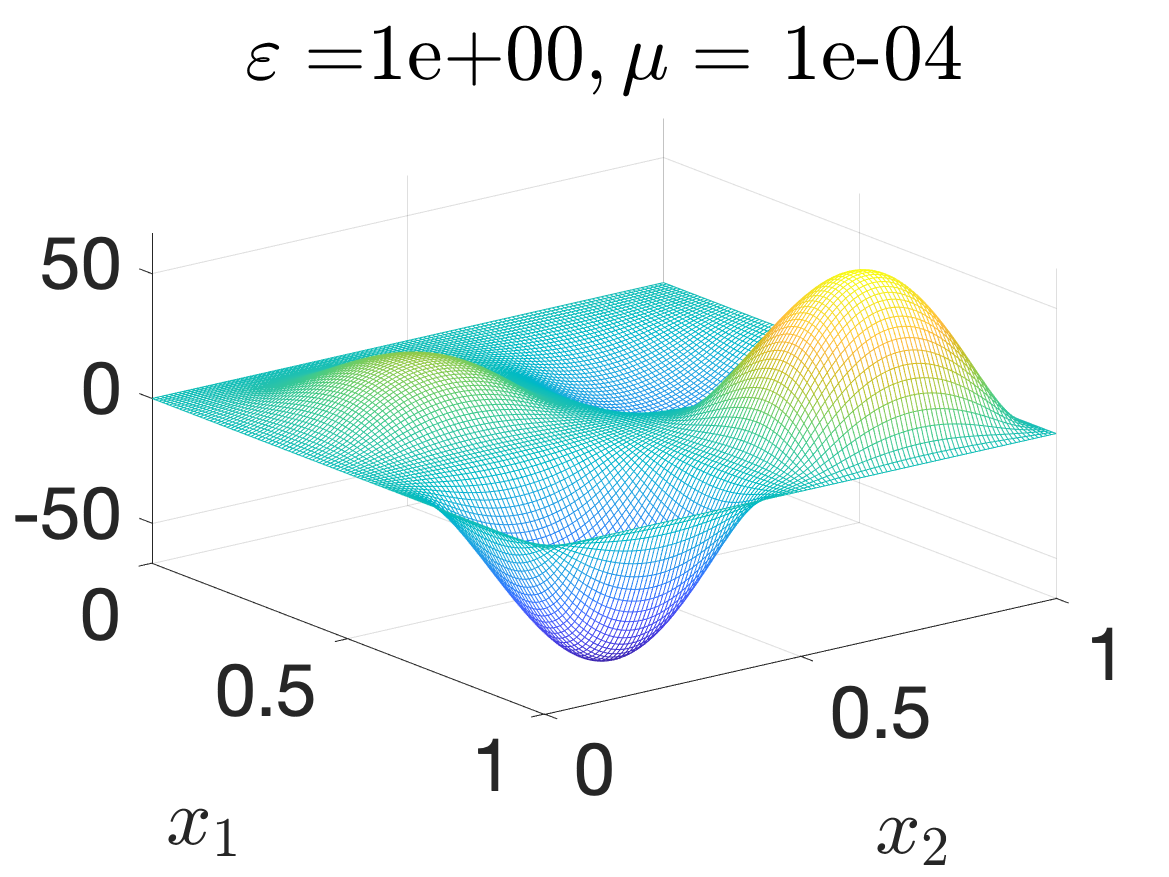}\hfil
	\includegraphics[width=0.3\textwidth]{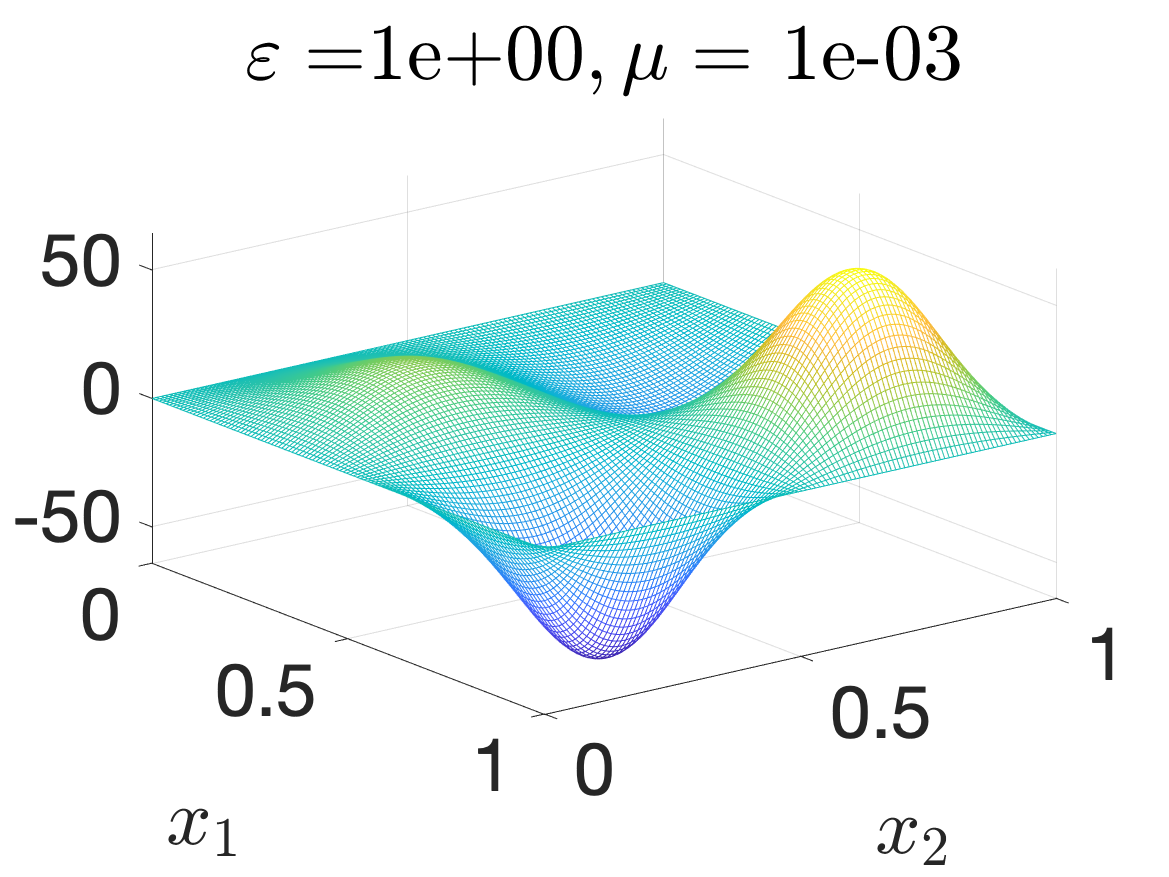}\\
	
	\includegraphics[width=0.3\textwidth]{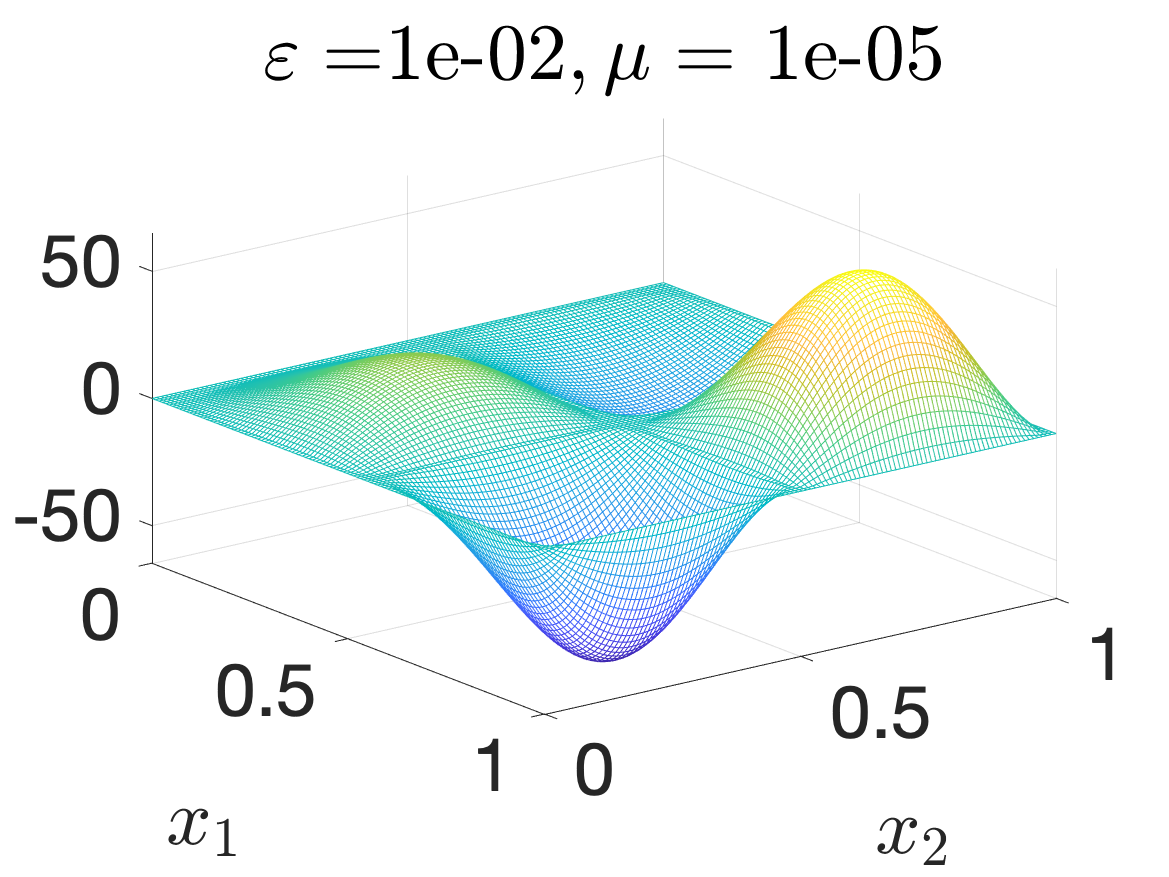}\hfil
	\includegraphics[width=0.3\textwidth]{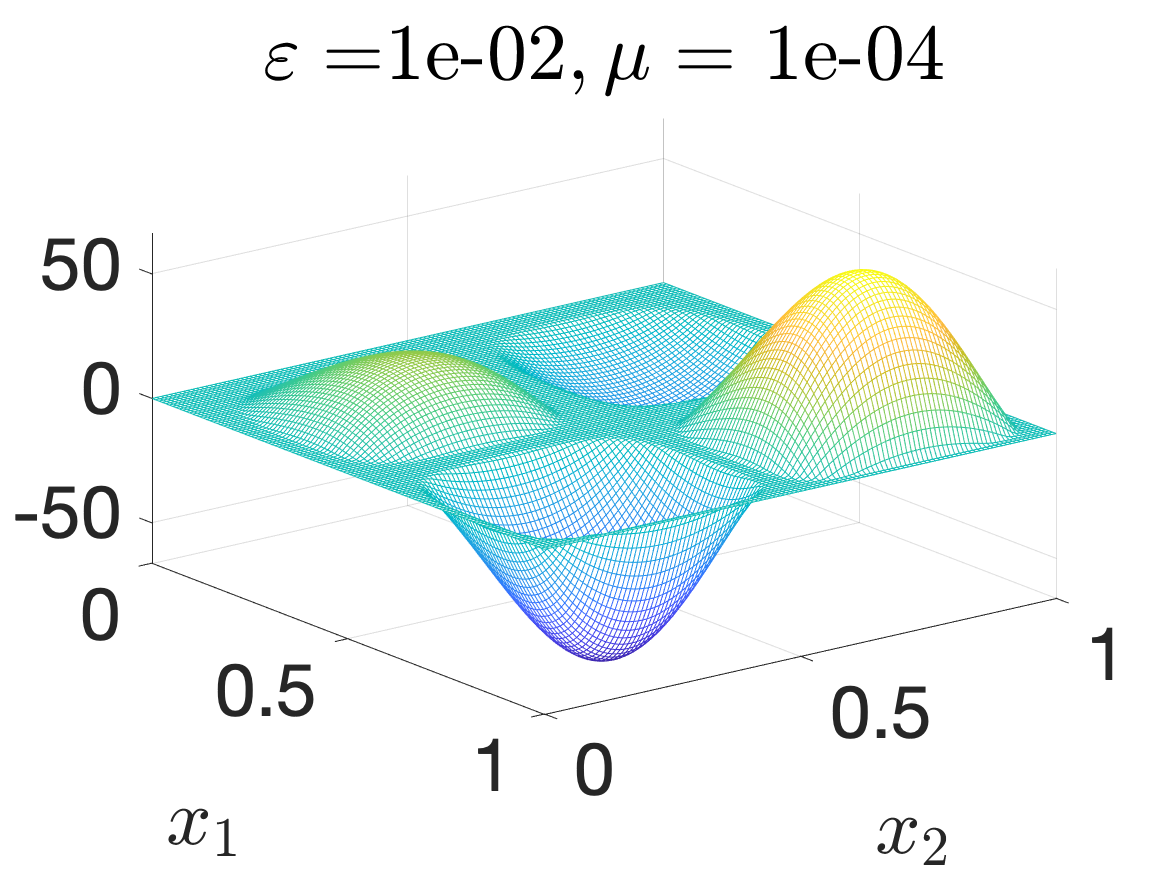}\hfil
	\includegraphics[width=0.3\textwidth]{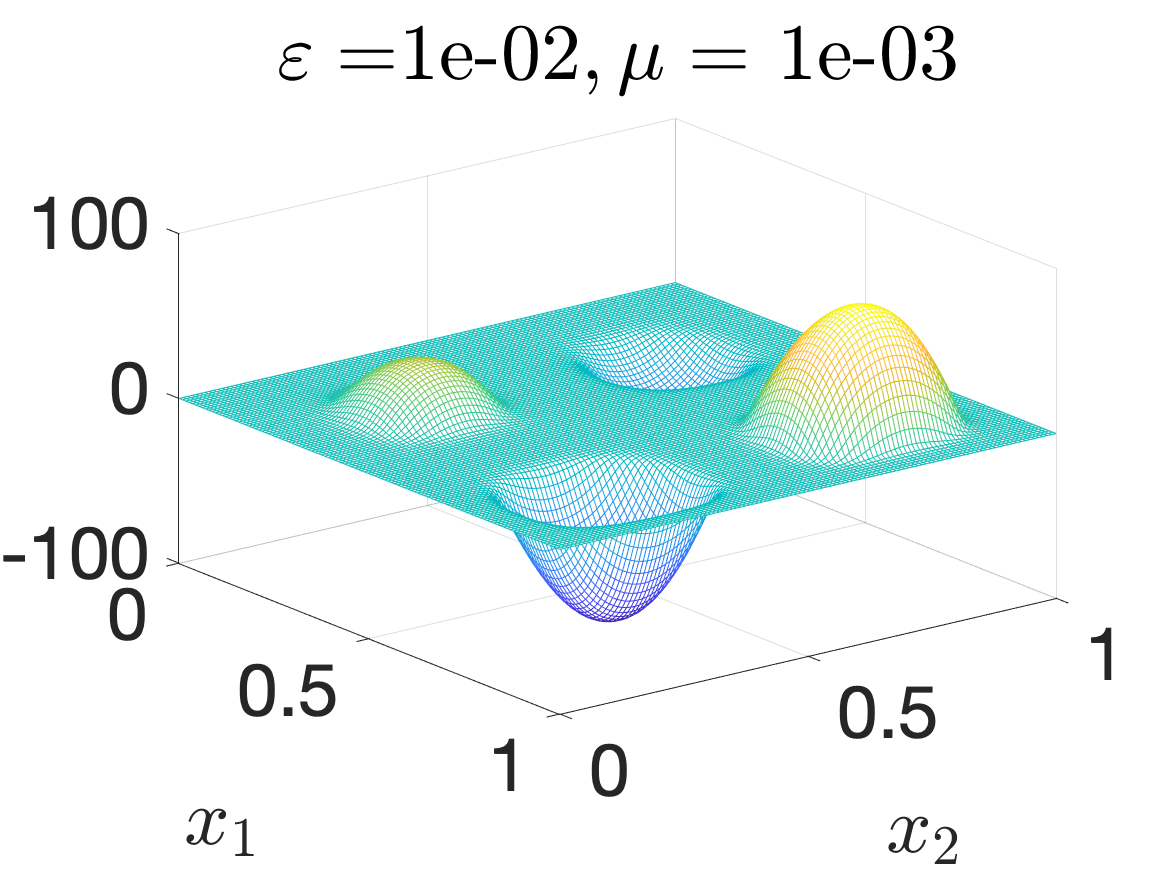}\\
	
	\includegraphics[width=0.3\textwidth]{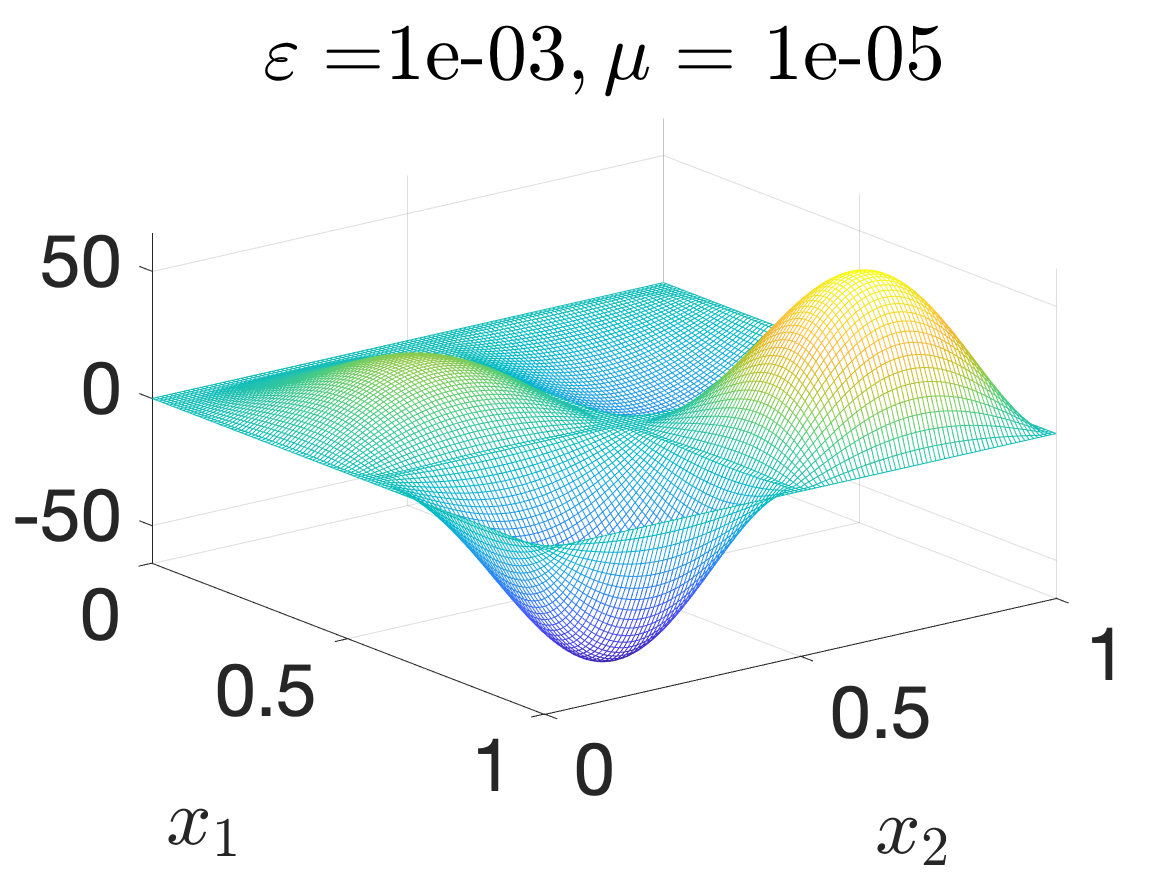}\hfil
	\includegraphics[width=0.3\textwidth]{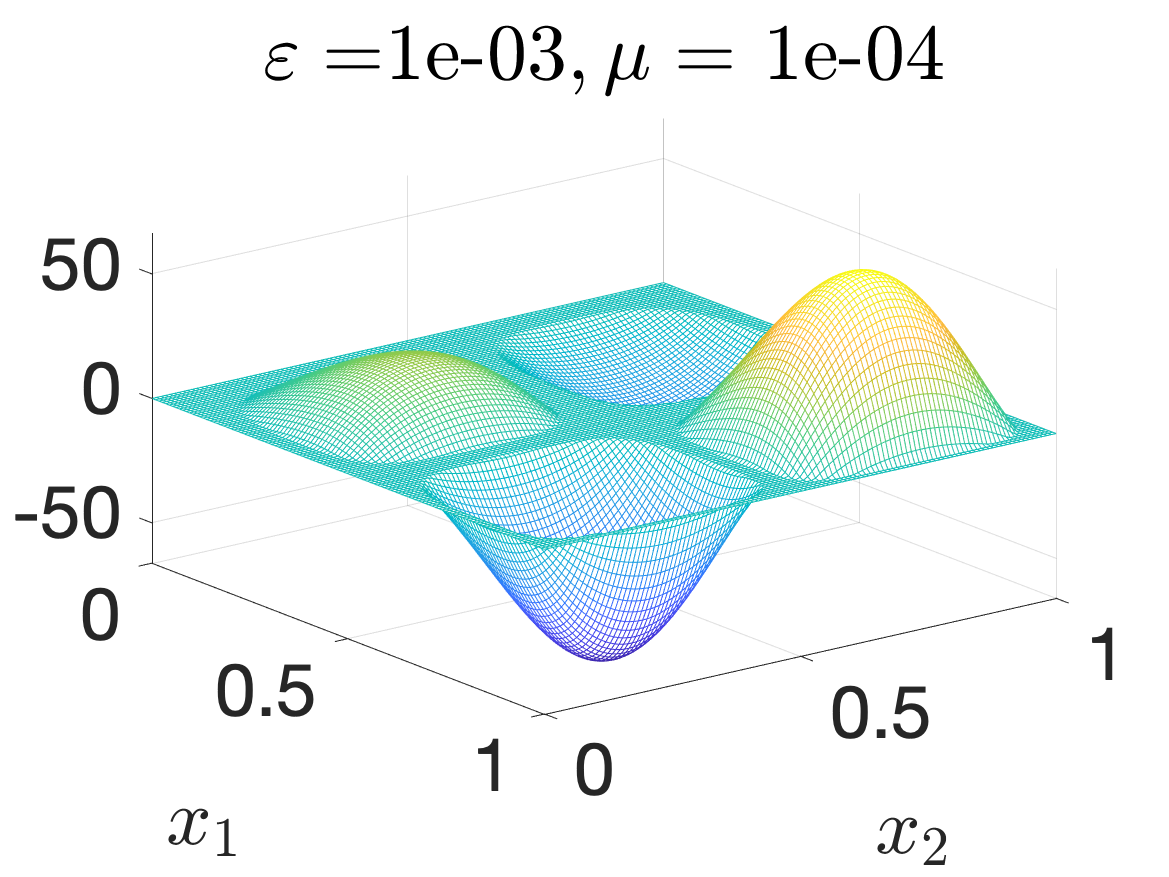}\hfil
	\includegraphics[width=0.3\textwidth]{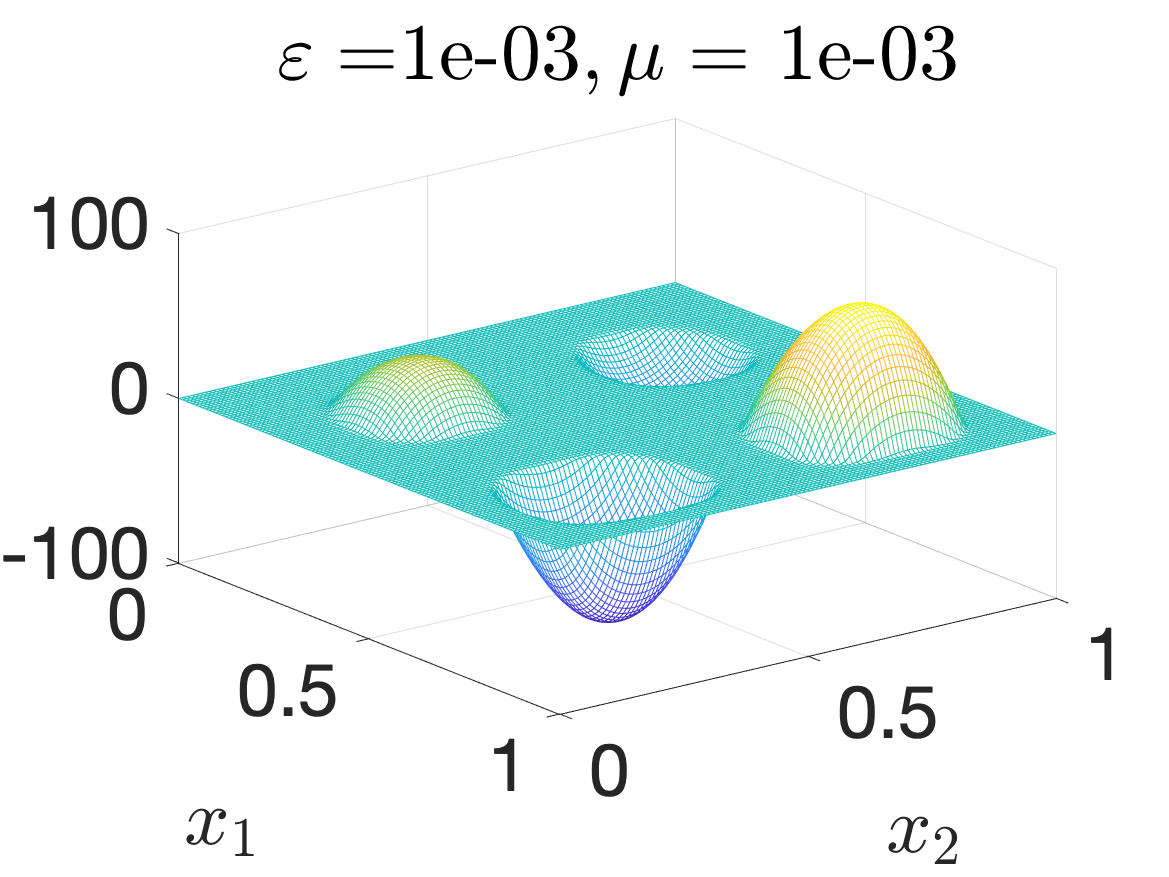}\\

 \includegraphics[width=0.3\textwidth]{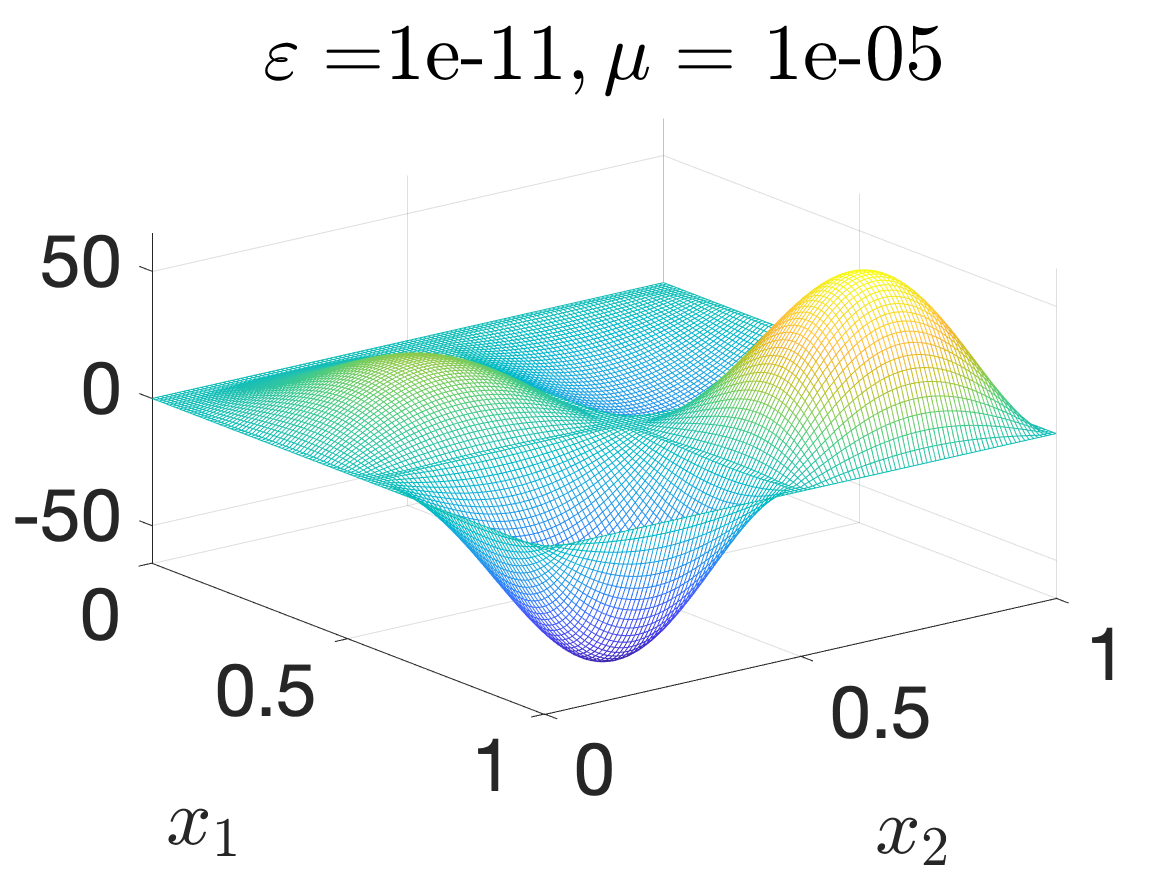}\hfil
	\includegraphics[width=0.3\textwidth]{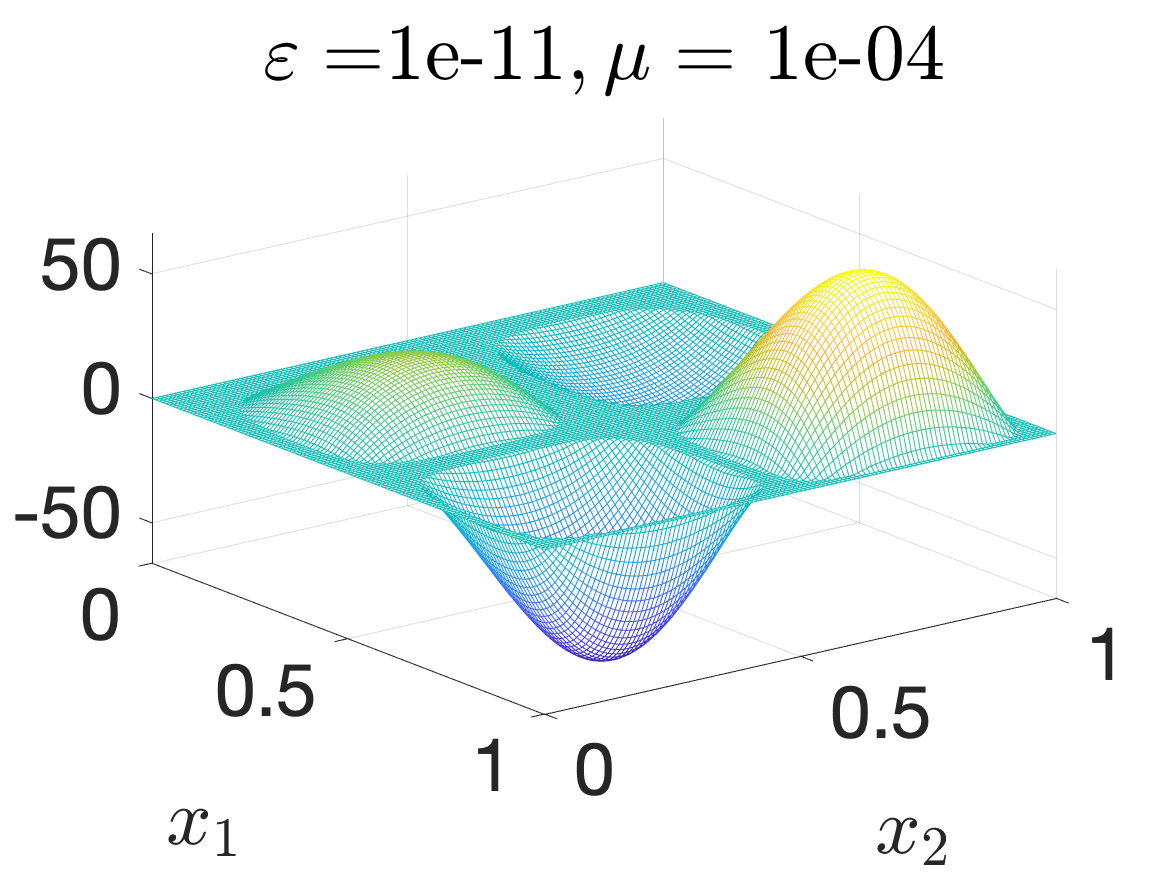}\hfil
	\includegraphics[width=0.3\textwidth]{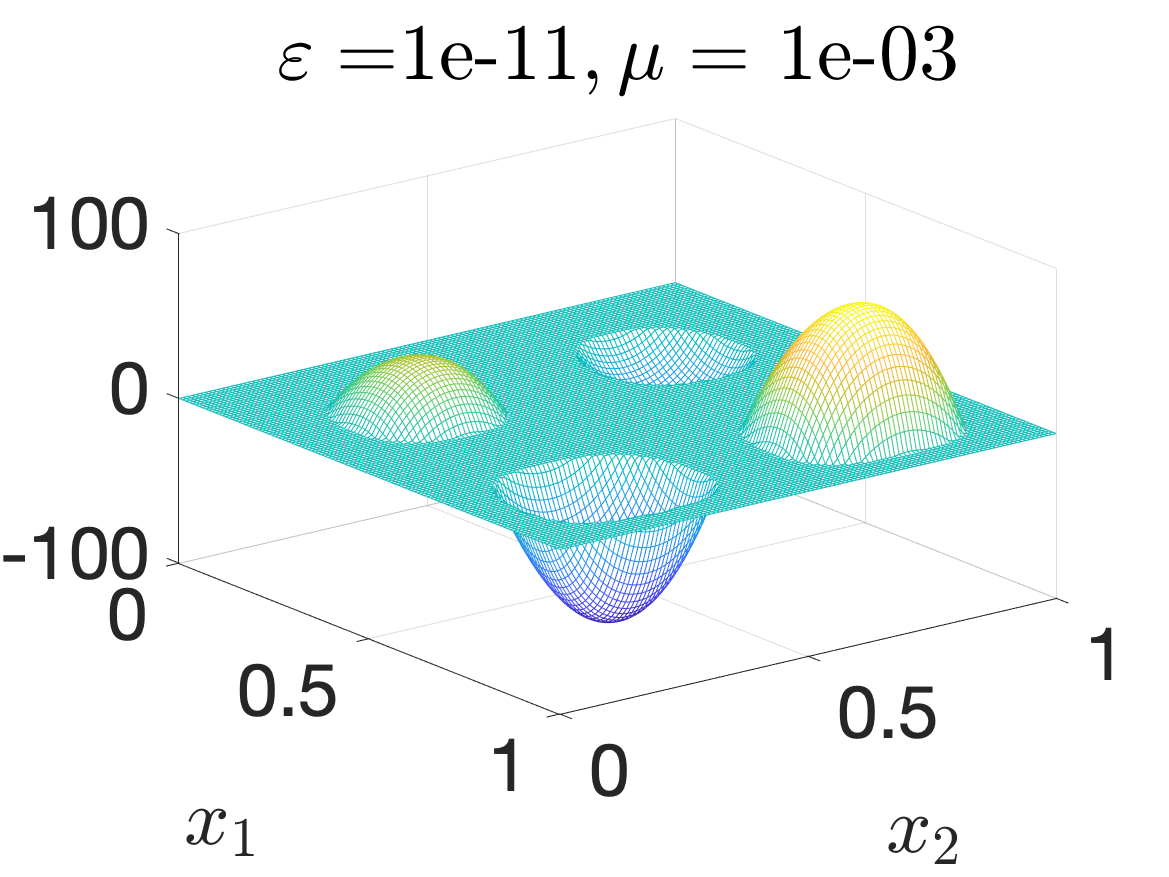}\\
	\caption[a]{Optimal control functions for test configuration \cite[Example 1]{Stadler:2009:1} corresponding to $\mu \in \{10^{-5},10^{-4},10^{-3} \}$ (from left to right)  and to $\varepsilon \in \{ 1,10^{-2},10^{-3},10^{-11} \ \}$ (from top to bottom).}\label{fig: sparsity test}
\end{figure}

\FloatBarrier

\section{Computational framework and numerical results}\label{sec:FrameworkAndNumerics}
In this section, we present the main components of our approach for finding solutions to the (smoothed) optimality system(s).
These are Newton methods, continuation strategies, and domain-decomposition (linear/nonlinear) preconditioning.
Specifically, in Subsection \ref{sec:NewtonAndContinuation}, we discuss the use of a damped Newton method for the (monolithic) solution of the smoothed system.
When the smoothing parameter $\varepsilon$ is small, the behavior to be expected from the damped Newton is the same as that of a damped semi-smooth Newton method to the unsmoothed system. 
Thus, we propose a continuation strategy in the smoothing parameter and address the benefits of augmenting the straightforward Newton approach by this technique.
This idea will be combined with a nonlinear preconditioning approach based on the RASPEN domain decomposition method in Subsection \ref{sec:NonPrec}.
In order to facilitate a fair comparison to a sophisticated computational framework without nonlinear preconditioning, we will employ the RAS method as a linear preconditioner for solving the linear systems within the monolithic Newton.

We will investigate the algorithmic and numeric performance of the combinations of these approaches.
We employ a model problem to examine the performance.
Specifically, we fix the unit square domain $\Omega=(0,1)^2$, the laplacian $A=-\Delta$ and the nonlinearity $\varphi(y)=\kappa (y^3+\exp(\kappa y))$.
The problem parameters are set to $\kappa=0.1$, $\nu= 10^{-6}$ and $ \mu=1$.
We fix $f\equiv0$, set
\begin{equation}\label{eq:ppp}
	\bar p(x_1,x_2) = 1.3\mu \sin(2\pi\tilde{k} x_1)\sin(2\pi \tilde{k} x_2),
\end{equation}
for $\tilde{k}=5$ and compute $\bar y$ as the solution of
\begin{equation}\label{eq:pppp}
	A\bar y + \varphi(\bar y)-f -\frac{1}{\nu}\bigg(\bar{p}+\mu P_\varepsilon\bigg(-\frac{\bar{ p}}{\mu}\bigg)\bigg) =0,
\end{equation}
for $\varepsilon = 10^{-15}$ and set $y_d$ as
\begin{equation}\label{eq:ydpppp}
	y_d = -Ap -\varphi'(y)\bar p + \bar y.
\end{equation}
The constructed quantities are depicted in Figure~\ref{fig: computed optimal test}. Proceeding this way guarantees that $\bar y$ and $\bar p$ are solutions to the first-order optimality system and by choosing $\bar p$, we can guarantee that the nonlinearity and non-differentiability in the projection operator become relevant, as $p$ crosses the thresholds of $-1$ and $1$ in various sections of the domain.
Our setting is discretized using finite differences with $N=450$ discretization points per dimension and P1 finite elements.

\begin{figure}[t]
	\centering
	\includegraphics[width=0.37\textwidth]{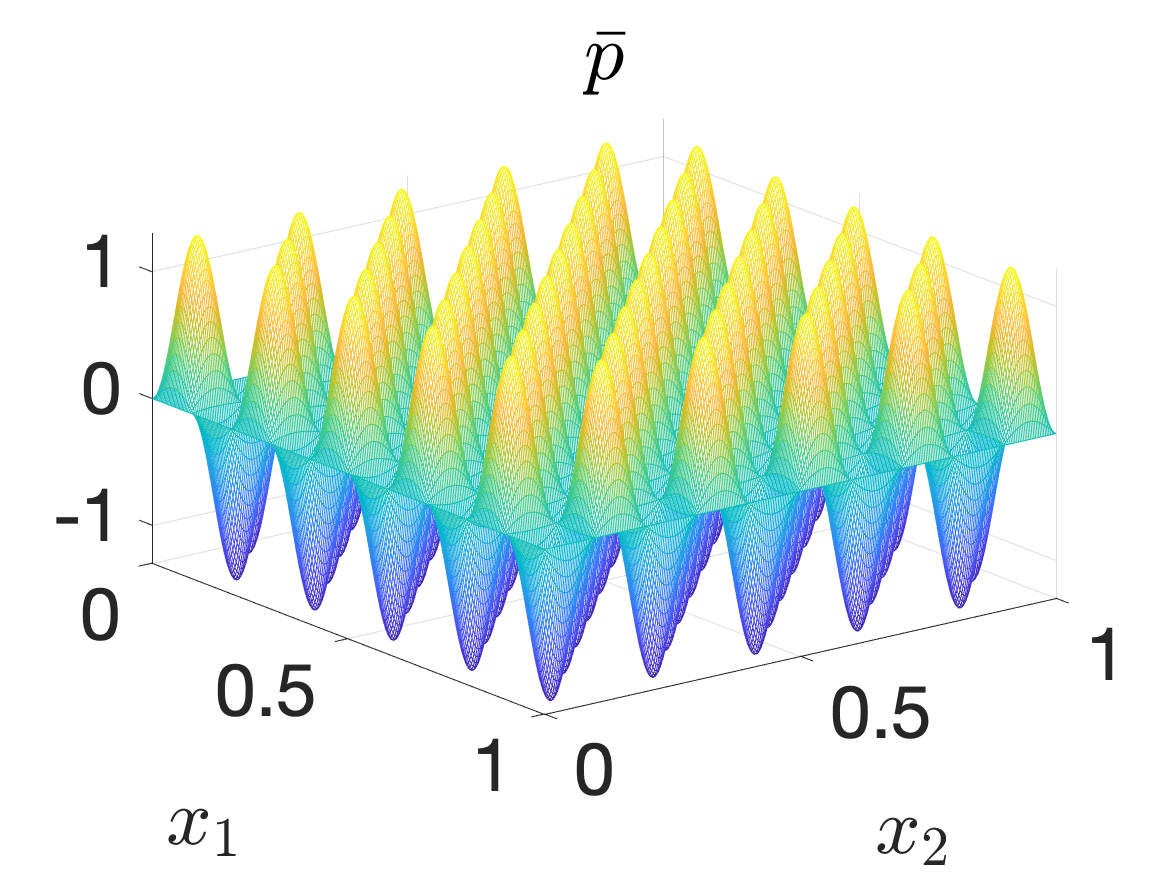}\hfil
	\includegraphics[width=0.37\textwidth]{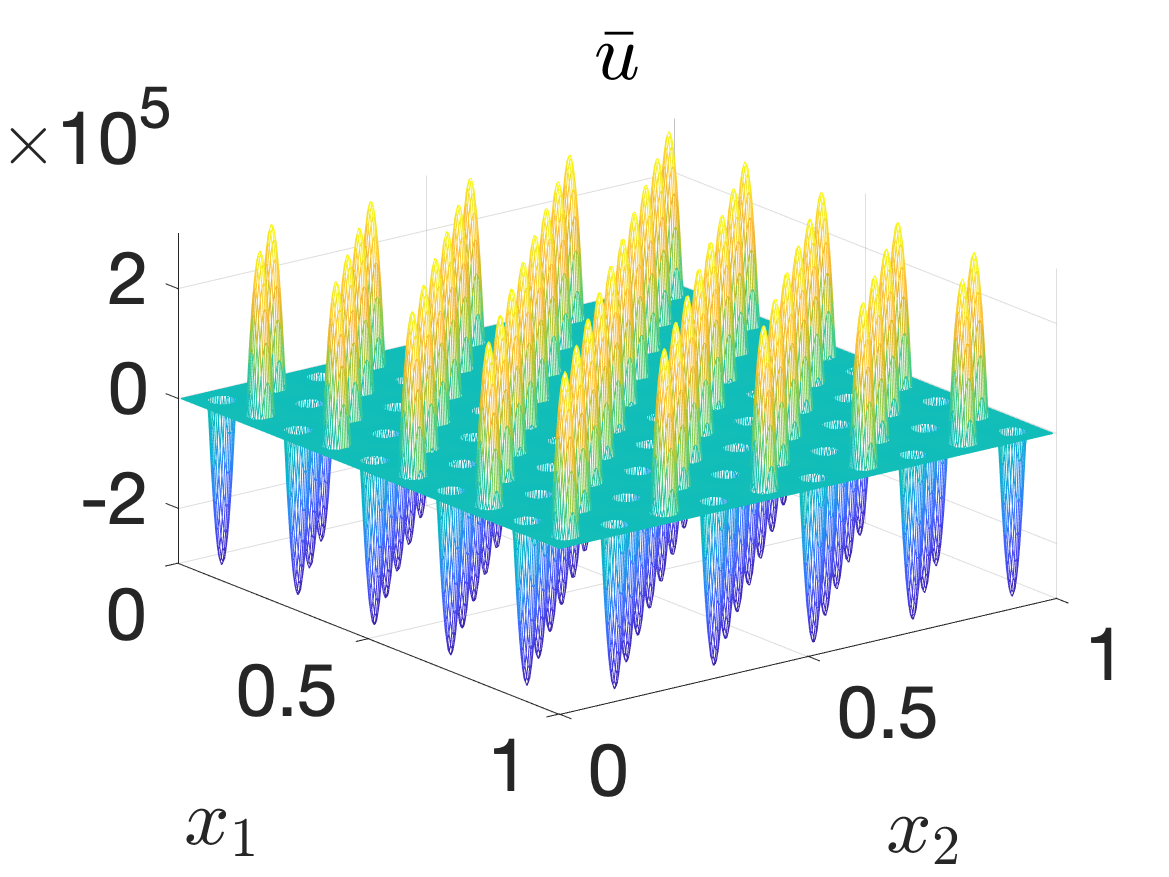}\\
	\includegraphics[width=0.37\textwidth]{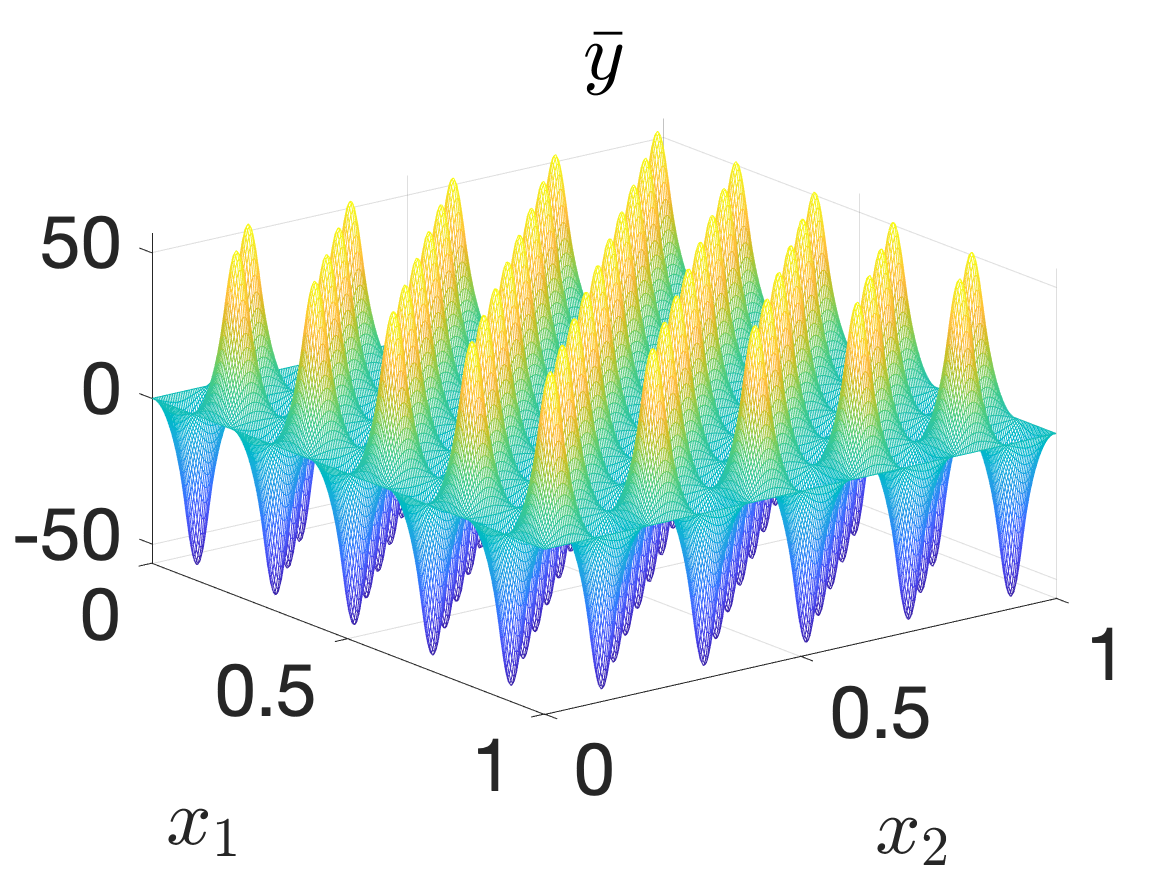}\hfil
	\includegraphics[width=0.37\textwidth]{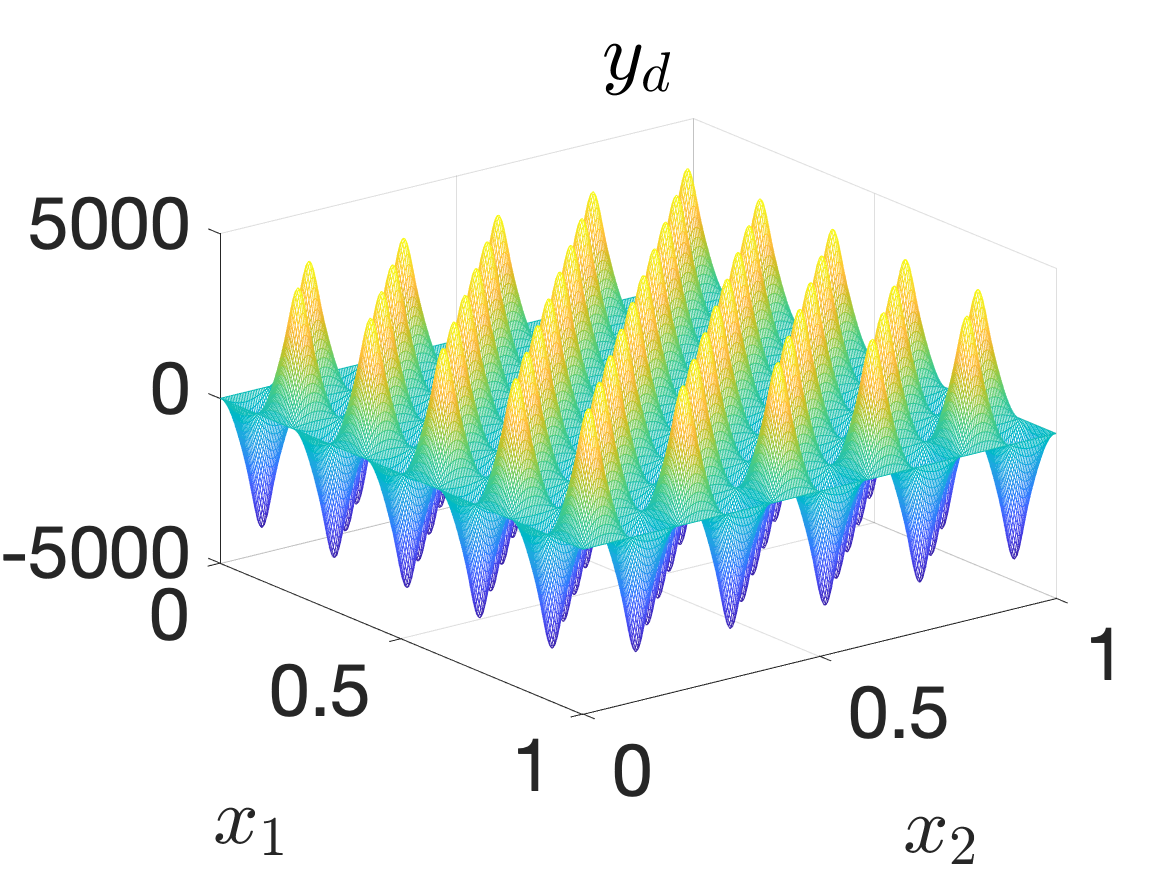}
	\caption[a]{The constructed optimal control $\bar u$, optimal state $\bar y$ and optimal adjoint state $\bar p$ and desired state $y_d$.}\label{fig: computed optimal test}
\end{figure}

\subsection{Damped Newton and continuation}\label{sec:NewtonAndContinuation}
In this section, we present a damped Newton method for the solution of the regularized optimality system \eqref{zeropoint} and a continuation strategy in the smoothing parameter $\varepsilon$. We will denote $F_\varepsilon(x) = F(\varepsilon, x)$ for $x\in V$ for the remainder of this paper.
Owing to the regularization, the map ${F}_\varepsilon$ is differentiable
and it is possible to use a classical damped Newton method to solve \eqref{zeropoint}.
Given an iterate $x^k$, the new approximation $x^{k+1}$ is obtained as
\begin{equation}\label{eq:Newton_step}
	x^{k+1}=x^k+\alpha_k \, d^k.
\end{equation}
Here, the direction $d_k$ is computed by solving the Newton system
\begin{equation}\label{linNewtonPrec}
	F_\varepsilon'(x^k)d^k=-{F}_\varepsilon(x^k).
\end{equation}
In \eqref{eq:Newton_step}, $\alpha_k \in (0,1]$ is a damping parameter
that is computed by a backtracking approach to satisfy the condition
\begin{equation}\label{eq:line_search}
	\|F_{{\varepsilon}}(x^k+\alpha_k\,d^k)\| \leq \sigma\|F_{{\varepsilon}}(x^k)\|,
\end{equation}
where $\sigma \geq 1$ is a relaxation parameter. Note that \eqref{eq:line_search} has the
the goal of avoiding excessively large growth of the residual value $F_{{\varepsilon}}(x^k+\alpha_k\,d^k)$ and is less restrictive than the requirement that it must decay
monotonically along the iterations.

The first and second rows of Table~\ref{tab:MonolithicNewtonAndContinuation} show the number of Newton iterations needed to solve \eqref{zeropoint} up to absolute or relative tolerance $\mbox{tol} = 10^{-10}$ for different values of the smoothing parameter $\varepsilon=\varepsilon_{\min}$.
Seeing as the number of required Newton iterations increases as $\varepsilon$ decreases, it is apparent that the smoothing has a regularizing effect.
This suggests that using a continuation approach on the regularization parameter $\varepsilon$ can be beneficial for the overall performance of the method.
Specifically, for a given target value $\varepsilon_{\min} \in (0,1]$ of the smoothing parameter, we modify the computations \eqref{eq:Newton_step}--\eqref{linNewtonPrec} in the Newton iteration by starting with a rather large initial smoothing parameter $\varepsilon_0=1$ and successively reducing $\varepsilon$ at each iteration using the update $\varepsilon_{k+1} = \max\{\gamma \varepsilon_k, \varepsilon_{\min}\}$, where $\gamma \in(0,1)$ is a parameter controlling the rate at which the sequence $(\varepsilon_k)_k$ decays towards $\varepsilon_{\min}$. This procedure is summarized in Algorithm~\ref{alg:epsilonContinuation}.
Note that if $\varepsilon_0=\varepsilon_{\min}$ is chosen, then no continuation is performed and Algorithm~\ref{alg:epsilonContinuation} corresponds exactly to the Newton method \eqref{eq:Newton_step}-\eqref{linNewtonPrec} applied to the system \eqref{zeropoint} with $\varepsilon = \varepsilon_0$.
\begin{algorithm}[t]
	\caption{Monolithic Newton with $\varepsilon$-continuation and relaxed backtracking linesearch}
	\begin{algorithmic}[1]
		\Require 
		Initial guess $x^0$
		,
		backtracking relaxation $\sigma \geq 1$
		,
		initial value $\varepsilon_0 > 0$
		,
		target regularization $\varepsilon_{\min}$
		,
		continuation parameter $\gamma \in (0,1]$,
		tolerance of convergence $\mbox{tol}>0$.
		
		\State Set $k=0$. 
		\While{$\|F_{{\varepsilon_k}}(x^k)\|> \max\{\mbox{tol},{\mbox{tol}}{\|F_{{\varepsilon_0}}(x^0)\|}\}$ and $k\leq k_{\max}$}
		\State Solve $F'_{{\varepsilon_k}}(x^k)d^k=-F_{{\varepsilon_k}}(x^k)$. \label{alg:epsilonContinuation:linSysSolve}
		\State Set $\alpha_k=1 $.
		\While{$\|F_{{\varepsilon_k}}(x^k+\alpha_kd)\|>\sigma\|F_{{\varepsilon_k}}(x^k)\|$}
		\State $\alpha_k=\frac{\alpha_k}{2}$.
		\EndWhile
		\State Update $x^{k+1} = x^{k}+\alpha_k \, d^{k}$,
		$\varepsilon_{k+1} = \max\{\gamma \varepsilon_k, \varepsilon_{\min}\}$, and $k = k+1$.
		\EndWhile
	\end{algorithmic}
	\label{alg:epsilonContinuation}
\end{algorithm}
\begin{table}[t!]
\centering
\begin{tabular}{||l|| c|c|c|c|c|c||}
\hline
Method$\backslash$ $\varepsilon_{\min}$ & 1&1e-3 & 1e-5&1e-10& 1e-13 &  1e-15  \\
\hline\hline
Newton &  11 & 31 & 35 & 41 & 40 & 40     \\
Newton$_\varepsilon$ &  11 & 22 & 20 & 21 & 21 & 23   \\
\hline  \hline
\end{tabular}
\caption{Iterations of monolithic Newton method (Newton) with fixed values $\varepsilon=\varepsilon_{\min}$ and monolithic Newton method with continuation (Newton$_\varepsilon$) starting from $\varepsilon = 1$ down to $\varepsilon_{\min}$ according to Algorithm \ref{alg:epsilonContinuation}.}
\label{tab:MonolithicNewtonAndContinuation}
\end{table}
The computational cost of one iteration of Algorithm \ref{alg:epsilonContinuation} is dominated by the cost of solving the linear system \eqref{linNewtonPrec} in Step \ref{alg:epsilonContinuation:linSysSolve}.
An efficient approach to solving the system is using iterative Krylov methods like MINRES or more generally GMRES (see, e.g., \cite{ciaramella2022iterative}), which was employed as the solver in the results of Table~\ref{tab:MonolithicNewtonAndContinuation}. In order to improve the performance of GMRES, we incorporate a RAS preconditioner. Table~\ref{tabl: linear_prec_GMRES} reports the average number of GMRES iterations, with and without
the use of the RAS preconditioner, corresponding to the same problem solved in Table~\ref{tab:MonolithicNewtonAndContinuation}.
While a side effect of the continuation strategy appears to be a minimal reduction of the number of average GMRES iterations, the effect is obviously much larger for the RAS preconditioning.
\begin{table}[t!]
\centering
\begin{tabular}{||l|| c|c|c|c|c|c||}
\hline
Method$\backslash$ $\varepsilon_{\min}$ & 1&1e-3 & 1e-5&1e-10& 1e-13 &  1e-15  \\
\hline\hline
Newton & 866 & 1247 & 1261 & 1328 & 1341 & 1341       \\
Newton$_\varepsilon$ & 866 & 1166 & 1153 & 1162 & 1167 & 1128   \\ 
Newton$_\text{RAS}$&  28 & 33 & 32 & 33 & 33 & 33\\
Newton$_{\text{RAS},\varepsilon}$&  28 & 32 & 32 & 32 & 32 & 31 \\
\hline  \hline
\end{tabular}
\caption{
Average GMRES iterations of
monolithic Newton (Newton),
monolithic Newton with continuation (Newton$_\varepsilon$),
(linearly) RAS-preconditioned monolithic (Newton$_\text{RAS}$)
and (linearly) RAS-preconditioned Newton with continuation (Newton$_{\text{RAS},\varepsilon}$ ) starting from $1$ to $\varepsilon_{\min}$ for a $2\times 2$-subdomain decomposition of $\Omega$.
}
\label{tabl: linear_prec_GMRES}
\end{table}
\subsection{Nonlinearly preconditioned Newton}\label{sec:NonPrec}
While linear domain decomposition preconditioners can be used to compute the update direction $d^k$
more efficiently, they do not generate better search directions. As a result, they cannot improve the performance of the Newton method e.g. when the initial guess is far from the solution. A related approach to accelerate and robustify the solution procedure with respect to initial guesses is to employ a nonlinear preconditioner. In this approach, one directly transforms the nonlinear system \eqref{zeropoint} and then applies Newton's method on the new transformed problem.
This way, different search directions are obtained that generally improve the convergence behavior of Newton's method; see, e.g., \cite{DGKW16}.
Here, we will outline how to extend the nonlinear RAS preconditioner originally proposed for solving nonlinear PDEs in \cite{DGKW16} to the solution of our (regularized) optimal control problem.

In the Schwarz method for solving \eqref{zeropoint}, one begins with a non-overlapping decomposition of $\Omega$ into $I\in \N$ subdomains $\widetilde{\Omega}_i$, i.e., $\overline{\Omega} = \cup_{i=1}^I\overline{\widetilde{\Omega}}_i$.
Each non-overlapping subdomain $\widetilde{\Omega}_i$ is enlarged by an overlap
to obtain a new subdomain $\Omega_i$ containing $\widetilde{\Omega}_i$.
The subdomains $\Omega_i$ give rise to an overlapping decomposition:
$\Omega = \cup_{i=1}^I \Omega_i$.
Thus, given an initial guess $x_i^0=(y^0_i, p_i^0)\in V_i\coloneqq H^1(\Omega_i)^2$, one iteratively solves the 
weak form of the local subproblems
\begin{equation}\label{eq:RAS_reg_OS_red}
\begin{aligned}
{A} {y^k_i}+\varphi({y^k_i})&=f-\frac{1}{\nu}\bigg({p^k_i}+\mu P_\varepsilon\bigg(-\frac{{p^k_i}}{\mu}\bigg)\bigg) &\text{ in }
\Omega_i, \\
\Big({A} +\varphi'({y^k_i})\Big){p^k_i}&={y^k_i}-y_d &\text{ in }   \Omega_i, \\
y^k_i&=p^k_i = 0 &\text{ on }  \partial\Omega_i\cap \partial\Omega,\\
{y^k_i}&={y^{k-1}_j}, \ {p^k_i}={p^{k-1}_j}&\text{ on }  \
\partial\Omega_i\cap \Omega_j \  (j\neq i),
\end{aligned}
\end{equation}
on the subdomains $\Omega_i$ yielding $x_i^k= (y_i^k,p_i^k) \in V_i$.
The approximation $x_k$ in the entire domain $\Omega$ is obtained as the recombination $x^k = \sum_{i=1}^I \widetilde{P}_i x_i^k$ with the prolongation operators $\widetilde{P}_i\colon V_i \to L^2(\Omega)^2$ defined, for any $v=(v_y,v_p) \in V_i$ and $w \in H_1(\Omega_i)$, as
\begin{equation*}
\widetilde{P}_i(v)
\coloneqq
\begin{bmatrix}
\widetilde{\mathbb{P}}_i(v_y)\\
\widetilde{\mathbb{P}}_i(v_p)
\end{bmatrix}
\qquad
\text{ with }
\qquad
\widetilde{\mathbb{P}}_i(w)
\coloneqq
\begin{cases}
w &\text{a.e. in $\widetilde{\Omega}_i$},\\
0 &\text{otherwise}. \\
\end{cases}
\end{equation*}
To obtain an abstract version of the weak form of \eqref{eq:RAS_reg_OS_red}, we introduce the prolongation operator $P_i\colon V_i \to L^2(\Omega)^2$ defined, for any $v=(v_y,v_p) \in V_i$ and $w \in H^1(\Omega_i)$, as
\begin{equation*}
P_i(v)
\coloneqq
\begin{bmatrix}
\mathbb{P}_i(v_y)\\
\mathbb{P}_i(v_p)
\end{bmatrix}
\qquad \text{ with } \qquad
\mathbb{P}_i(w) \coloneqq
\begin{cases}
w &\text{a.e. in $\Omega_i$},\\
0 &\text{otherwise}, \\
\end{cases}
\end{equation*}
and the restriction operator $R_i\colon L^2(\Omega)^2 \to L^2(\Omega_i)$ defined, for any $v=(v_y,v_p) \in L^2(\Omega)^2$ and $w \in L^2(\Omega)$, as
\begin{equation*}
R_i(v)
\coloneqq
\begin{bmatrix}
\mathbb{R}_i(v_y) \\
\mathbb{R}_i(v_p)
\end{bmatrix}
\qquad \text{ with } \qquad
\mathbb{R}_i(w) \coloneqq w|_{\Omega_i}.
\end{equation*}
Note that $R_i$ maps $H^1(\Omega)$ into $H^1(\Omega_i)$ and that $R_i P_i = I_{V_i}$ for $i=1,\ldots,I$
and $\sum\limits_{i=1}^{I}\widetilde{P}_i R_i=I_{V}$,
where $I_{V_i}$ and $I_{V}$ are identity operators.

Now, given any pair $\widehat{x}$, $\widetilde{x}$ in $V$, it is clear that
\begin{equation*}
P_i R_i \widehat{x} + (I_V - P_i R_i) \widetilde{x} =
\begin{cases}
\widehat{x}|_{\Omega_i} &\text{in $\Omega_i$}, \\
\widetilde{x}|_{\Omega\setminus \Omega_i} &\text{in $\Omega\setminus \Omega_i$}. \\
\end{cases}
\end{equation*}
Accordingly, for any Nemytskii operator associated to a function $\psi\colon \mathbb{R}^2 \rightarrow \mathbb{R}^2$, we have that
\begin{equation*}
R_i\psi(P_iR_i \widehat{x} + (I_V - P_i R_i) \widetilde{x}) = R_i \psi(P_i R_i \widehat{x}).
\end{equation*}
\
Hence, letting $T\colon V' \rightarrow V$ denote the canonical Riesz representation map, a direct calculation shows that
\begin{equation*}
R_i T F_\varepsilon(P_iR_i x^{k} + (I_V - P_i R_i) x^{k-1})
= R_i T F_\varepsilon(P_iR_i x^{k}) + R_i T
\begin{bmatrix}
A & 0 \\
0 & A \\
\end{bmatrix} (I_V - P_i R_i) x^{k-1},
\end{equation*}
where it is clear that, as in \eqref{eq:RAS_reg_OS_red},
$x^{k-1}$ is not affected by the (nonlinear) functions
$\varphi$ and $P_\varepsilon$, but only by the operator $A$.
Thus, the second term in the right-hand side of the above equation represents a weak formulation of the transmission condition in \eqref{eq:RAS_reg_OS_red} written in the residual form.
Accordingly, the weak form of \eqref{eq:RAS_reg_OS_red} can be written as
\begin{equation}\label{eq:RASPEN_weak}
R_i T F_\varepsilon(P_i R_i x^{k} + (I_V - P_i R_i) x^{k-1}) = 0,
\end{equation}
which we assume to be well-posed in the sense that there exists an
$x^{k} \in V$ such that $P_i R_i x^{k} + (I_V - P_i R_i) x^{k-1} \in V$.

Now, we denote by $C_i(x^{k-1})\in V_i$, for $i=1,\dots,I$, the solutions
to the subproblems \eqref{eq:RASPEN_weak}, i.e., they satisfy
\begin{equation}\label{eq: c_i}
\corr_i(C_i(x^{k-1}))\coloneqq R_i T F_\varepsilon(P_i C_i(x^{k-1}) + (I_V - P_i R_i) x^{k-1}) = 0.
\end{equation}
The $C_i(x^{k-1})$ are the local corrections of the Schwarz iteration that can be computed in parallel and that are used to obtain the new approximation as the recombination
\begin{equation}\label{eq:non_RAS}
x^k = x^{k-1} + \sum\limits_{i=1}^{I}\widetilde{P}_i C_i(x^{k-1}),
\end{equation}
yielding a RAS-type fixed-point iteration.\footnote{Note that at the discrete level (using, e.g., finite differences of $P^1$ finite elements) it is possible to obtain an equivalence between parallel Schwarz method iterations \eqref{eq:RAS_reg_OS_red} and the RAS residual form; see, e.g., \cite{Gander2008}.}
If this iteration converges, then the limit point $x$ satisfies
\begin{equation}\label{eq: RASPEN def}
\pre_{\varepsilon}'(x)=\sum\limits_{i=1}^{I}\widetilde{P}_i C_i(x) = 0.
\end{equation}
This equation is the RAS preconditioned form of the original smoothed problem \eqref{zeropoint}, and solving it is equivalent to solving \eqref{zeropoint} directly.
Newton's method applied to \eqref{eq: RASPEN def} is called one-level RASPEN.

The Newton routine requires the computation of the Jacobian of $\pre_{\varepsilon}$.
Using \eqref{eq: RASPEN def}, we get that
\begin{align}
\pre_{\varepsilon}'(x)&=\sum\limits_{i=1}^{I}\widetilde{P}_i C_i'(x). \label{eq: raspen der}
\end{align}
Thus, we compute the derivatives of $C_i(x)$, $i=1,...,I$, by differentiating \eqref{eq: c_i} in $x=x^{k-1}$ to obtain
\begin{align}
C_i'(x)=-\bigg(R_i T F_\varepsilon'\big(x^{(i)}\big)P_i\bigg)^{-1}R_i F_\varepsilon'\big(x^{(i)}\big)
&=-K_i'(C_i(x))^{-1}R_i F_\varepsilon'\big(x^{(i)}\big)\label{DerC}
\end{align}
with $x^{(i)} = x+P_i C_i(x)$, where we used that the Jacobian of $\corr_i$ with respect to $C_i(x)$ is
$\corr_i'(C_i(x))= R_i T F_\varepsilon'(x^{(i)})P_i$.
With the derivative \eqref{eq: raspen der}, one RASPEN step is given by solving the Newton system
\begin{equation}
\label{eq:raspen_newton_update}
\pre_{\varepsilon}'(x^{k-1}) d^k = - \pre_{\varepsilon}(x^{k-1})
\end{equation}
and updating the iterate via
\begin{equation}
x^{k} = x^{k-1} + d^k.
\end{equation}
The whole RASPEN procedure is detailed in Algorithm~\ref{RASPEN}. Note that we apply the $\varepsilon$-continuation strategy only for the solution of the inner problems \eqref{eq: c_i}, since in our numerical experiments RASPEN only needed a few outer iterations to converge (see Table \ref{tab:RASPEN1} in Section \ref{sec:num:RASPEN}).
There are two parts dominating the computational cost of one RASPEN iteration.
The first is the evaluation of $\pre_\varepsilon(x^{k-1})$ via \eqref{eq: RASPEN def} (see Algorithm \ref{EvalF_e}), which means solving the small, local systems \eqref{eq: c_i} for $C_i(x^{k-1})$ in parallel by using a Newton-type solver on the inner level, e.g., Algorithm \ref{alg:epsilonContinuation}.
The second part is solving the Newton linear system \eqref{eq:raspen_newton_update}.
This can be done efficiently using a (matrix-free) Krylov subspace method.
The corrections  $C_i(x^{k-1})$ are stored and used again for the assembly of the function $d\mapsto \pre'(x^{k-1})d$ (see Algorithm~\ref{EvalRaspenJac}).
\begin{algorithm}[t!]
\caption{One-level RASPEN with $\varepsilon$-continuation in the inner Newton}
\begin{algorithmic}[1]
\Require initial guess $x^0$, tolerance $\mbox{tol}$, maximum number of iterations $k_{\max}$, target regularization $\varepsilon$.
\State Initialize $k=0$.
\State Assemble $\mathcal{F}_{\varepsilon}(x^{0})$ via Algorithm \ref{EvalF_e} for $\varepsilon_{\min}=\varepsilon$ and store the corrections $\big(C_i(x^{0})\big)_{i=0}^I$.
\While {$\|\mathcal{F}_{\varepsilon}(x^k)\|>\max\{\mbox{tol},\mbox{tol}\|\mathcal{F}_{\varepsilon}(x^0)\}$ and $k< k_{\max}$}
\State Compute $d^k$ by solving $\mathcal{F}_{\varepsilon}'(x^{k})d^k=-\mathcal{F}_{\varepsilon}(x^{k})$ via a matrix-free Krylov method, where
the map $(d\mapsto \mathcal{F}_{\varepsilon}'(x^{k})d)$ is assembled using $\big(C_i(x^{k})\big)_{i=0}^I$ in Algorithm \ref{EvalRaspenJac}.
\State Update $x^{k+1}=x^{k}+d^k$.
\State Set $k=k+1$.
\State Assemble $\mathcal{F}_{\varepsilon}(x^{k})$ by Algorithm \ref{EvalF_e} for $\varepsilon_{\min}=\varepsilon$ and store the corrections $\big(C_i(x^{k})\big)_{i=0}^I$.
\EndWhile
\State \textbf{Output:} $x^k$.
\end{algorithmic}
\label{RASPEN}
\end{algorithm}
\begin{algorithm}[t]
\caption{Evaluation of $\mathcal{F}_{\varepsilon}$: $x\to \mathcal{F}_{\varepsilon}(x)$}
\begin{algorithmic}[1]
\Require iterate $x$,  target regularization $\varepsilon_{\min}$.
\For{$i=1,\ldots,I$ \textbf{in parallel}}
\State Solve the local systems 
(\ref{eq: c_i}) for $C_i(x)$ using Algorithm~\ref{alg:epsilonContinuation} with or without continuation up to $\varepsilon_{\min}$.
\EndFor
\State Assemble $\mathcal{F}_{\varepsilon}(x)$ using \eqref{eq: RASPEN def}.
\State \textbf{Output:} $\mathcal{F}_{\varepsilon}(x), \big(C_i(x)\big)_{i=1}^I$.
\end{algorithmic}
\label{EvalF_e}
\end{algorithm}
\begin{algorithm}[t!]
\caption{Action of $\mathcal{F}_{\varepsilon}'(x):$ $d\to\pre'_{\varepsilon}(x)d$}
\begin{algorithmic}[1]
\Require iterate $x$, direction $d$,  corrections $\big(C_i(x)\big)_{i=1}^I$.
\For{$i=1,\ldots,I$ \textbf{in parallel}}
\State Solve the linear system arising from \eqref{DerC} for $C_i'(x)d$.
\EndFor
\State Assemble $\mathcal{F}_\varepsilon' (x)d$ via \eqref{eq: raspen der}.\\
\State \textbf{Output:} $\mathcal{F}_\varepsilon'(x)d$.
\end{algorithmic}
\label{EvalRaspenJac}
\end{algorithm}
%
In every GMRES iteration in solving the global linear system of the RASPEN iterations, the action of $d\mapsto \pre'(x^k)d$ is needed, so according to \eqref{eq: raspen der}, we can solve 
for $C_i'(x^k)d$ in parallel (see Algorithm \ref{EvalRaspenJac}).
The local linear systems of the inner Newton procedure are small in size and can be solved using direct solvers for sparse matrices (we apply Matlab's mldivide operation).
\section{Numerical experiments}\label{sec:num}
%
In this section, numerical experiments are performed to assess the efficiency of the proposed computational framework.  We compare the following six methods: Monolithic Newton with and without $\varepsilon-$continuation (\New, \Newc) (see Algorithm \ref{alg:epsilonContinuation}), linear RAS preconditioned Newton with and without $\varepsilon-$continuation (\NlinRAS, \NlinRASc), nonlinear RAS preconditioned Newton with and without $\varepsilon-$continuation in the first inner iteration (\RASPEN{}, \RASPENc{}) (see Algorithm \ref{RASPEN}). We consider a parallel implementation of the proposed algorithms on a CPU with 64 cores. As a baseline for comparison, we consider Newton without $\varepsilon$-continuation for small $\varepsilon$ ($\varepsilon \approx \varepsilon_{\text{mach}}$), since this algorithm essentially behaves like a semismooth Newton method (i.e., the equivalent Primal-Dual Active Set Strategy \cite{HintermuellerItoKunisch:2002:1}). Note that for the RASPEN methods, we consider the continuation strategy only in the first inner Newton iterations, as the bulk of the computation time for the inner Newton is concentrated there.
In particular, in Section \ref{sec:num:Mono}, we study the performance of the monolithic Newton method and the effect of our continuation strategy and linear RAS preconditioning.
This study provides important insights for the behavior of inner subdomain iterations of the nonlinear preconditioner (RASPEN), which is then studied in Section \ref{sec:num:RASPEN}. 
Further, a comparison of all presented methods is given in Section \ref{sec:num:RASPEN}.
All numerical tests are performed on problem \eqref{eq:P} with the settings reported in Section \ref{sec:FrameworkAndNumerics}.
Throughout the numerical experiments, we use an outer tolerance $tol=10^{-10}$ and for the inner Newton methods in RASPEN an inner tolerance of $10^{-8}$. The initial regularization is chosen to be $\varepsilon_0=1$ and the continuation rate as $\gamma=\frac{1}{5}$. 
We consider $N=450$ discretization points per dimension, leading to a system of size $2N^2= 405000$. The initial guess for all experiments is $x_0= 0\in \mathbb R^{2N^2}$ and we choose a backtracking parameter of $\sigma=1.1$. The overlap for the domain decomposition is set to $mh$ for $m=2$ and the mesh-size $h = \frac{1}{N+1}$.
\subsection{Monolithic Newton and linear preconditioning}\label{sec:num:Mono}
This section is concerned with numerical experiments to assess the performance of the monolithic Newton method as a baseline and the effect of both regularization/continuation and linear preconditioning.
To this purpose, we first set a $2 \times 2$ subdomain decomposition and report in Table~\ref{tab:MONO1} number of Newton iterations, (average) number of GMRES iterations and computational times (in seconds) of four different configurations: \New, \Newc, \NlinRAS\, and \NlinRASc\, for different values of final continuation values $\varepsilon_{\min}$.
\setlength\tabcolsep{1.6mm}
\begin{table}[t!]
	\centering
	\begin{tabular}{|l||c|c|c|c|c|}
		\hline
		$\varepsilon_{\min}$ & 1& 1e-5&1e-10& 1e-15  \\
		\hline\hline
		\New & 11 - 866 - 1.5e4 &35 - 1261 - 1.1e5 &41 - 1328 - 1.4e5 &40 - 1341 - 1.5e5  \\
		\Newc & 11 - 866 - 1.3e4 &20 - 1153 - 4.6e4 &21 - 1162 - 5.0e4 &23 - 1128 - 5.1e4\\
		\NlinRAS & 11 - 28 - 2.7e2 &35 - 32 - 9.9e2 &41 - 33 - 1.2e3 &40 - 33 - 1.2e3 \\
		\NlinRASc & 11 - 28 - 2.7e2 &20 - 32 - 5.5e2 &21 - 32 - 5.9e2 &23 - 31 - 6.4e2   \\
		\hline 
	\end{tabular}
	\caption{Outer Newton iterations - average GMRES iterations - computational times [s] of four configurations of monolithic Newton for different values of $\varepsilon=\varepsilon_{\min}$ and a $2\times2$-subdomain decomposition.}
	\label{tab:MONO1}
\end{table}
The results of Table~\ref{tab:MONO1} show clearly the beneficial effect of both linear preconditioning and continuation. On the one hand, RAS linear preconditioning impacts only the number of GMRES iterations, reducing them by a factor of about 10. On the other hand, the continuation strategy is capable of reducing substantially the number of outer Newton iterations (by a factor of 2-3 for $\varepsilon$ equal to $10^{-5}$, $10^{-10}$, and $10^{-15}$), while also leading to a reduction of number of GMRES iterations (even for the linearly preconditioned case (\NlinRASc). 
All these beneficial effects are clearly visible in the computational times.
Next, we study the robustness of linear RAS preconditioner and continuation with respect to the number of subdomains. Therefore we decompose the domain $\Omega=(0,1)^2$ into $2 \times s$ overlapping subdomains, for $s=2,\dots,8$. In Table~\ref{tab:MONO2}, we report the average number of GMRES iterations for the three configurations \New, \NlinRAS\, and \NlinRASc\, and different values of $\varepsilon_{\min}$.
\setlength\tabcolsep{1.6mm}
\begin{table}[h]
	\centering
	\begin{tabular}{|c||c|c|c|c|}
		\hline
		\# sub/ $\varepsilon_{\min}$ &1& 1e-5& 1e-10 & 1e-15  \\
		\hline\hline
		$2\times 2$ &866 - 28 - 28&1247 - 33 - 32&1261 - 32 - 32&1328 - 33 - 32\\ 
		$2\times 3$ &866 - 34 - 34&1247 - 45 - 44&1261 - 43 - 43&1328 - 44 - 42\\ 
		$2\times 4$ &866 - 37 - 37&1247 - 55 - 52&1261 - 52 - 50&1328 - 53 - 50\\ 
		$2\times 5$ &866 - 33 - 33&1247 - 47 - 45&1261 - 44 - 43&1328 - 46 - 43\\ 
		$2\times 6$ &866 - 38 - 38&1247 - 56 - 53&1261 - 53 - 51&1328 - 54 - 51\\ 
		$2\times 7$ &866 - 40 - 40&1247 - 60 - 57&1261 - 56 - 54&1328 - 58 - 54\\ 
		$2\times 8$ &866 - 40 - 40&1247 - 66 - 62&1261 - 62 - 59&1328 - 64 - 59\\ 
		\hline
	\end{tabular}
	
	\caption{Average GMRES iterations for \New - \NlinRAS - \NlinRASc, $N = 450$, and different subdomain decompositions.}
	\label{tab:MONO2}
\end{table}
We observe that the number of GMRES iterations increases with the number of subdomains especially for small regularization parameters.
Moreover, the number of GMRES iterations grow also with respect to $\varepsilon_{\min}$.
The beneficial effect of both linear RAS preconditioner and continuation is evident.
The computational times corresponding to the cases are reported in Table~\ref{tab:MONO3}.
\setlength\tabcolsep{1.6mm}
\begin{table}[h]
	\centering
	\begin{tabular}{|c||c|c|c|c|}
		\hline
		\# sub/ $\varepsilon_{\min}$ &1& 1e-5& 1e-10 & 1e-15  \\
		\hline\hline
		$2\times 2$ &1.5e4 - 2.7e2 - 2.7e2 &1e5 - 9.1e2 - 6.3e2 &1.1e5 - 9.9e2 - 5.5e2 &1.4e5 - 1.2e3 - 5.9e2\\ 
		$2\times 3$ &1.5e4 - 2.0e2 - 2.0e2 &1e5 - 7.3e2 - 5.1e2 &1.1e5 - 8e2 - 4.5e2 &1.4e5 - 9.6e2 - 4.6e2\\ 
		$2\times 4$ &1.5e4 - 1.6e2 - 1.6e2 &1e5 - 7.1e2 - 4.8e2 &1.1e5 - 7.4e2 - 4.2e2 &1.4e5 - 9.1e2 - 4.3e2\\ 
		$2\times 5$ &1.5e4 - 1.3e2 - 1.2e2 &1e5 - 5.3e2 - 3.4e2 &1.1e5 - 5.4e2 - 3e2 &1.4e5 - 6.7e2 - 3.2e2\\ 
		$2\times 6$ &1.5e4 - 1.7e2 - 1.4e2 &1e5 - 6.0e2 - 3.9e2 &1.1e5 - 6.1e2 - 3.3e2 &1.4e5 - 7.3e2 - 3.4e2\\ 
		$2\times 7$ &1.5e4 - 1.2e2 - 1.3e2 &1e5 - 5.8e2 - 4.3e2 &1.1e5 - 6.2e2 - 3.6e2 &1.4e5 - 7.6e2 - 3.7e2\\ 
		$2\times 8$ &1.5e4 - 1.3e2 - 1.3e2 &1e5 - 6.9e2 - 4.6e2 &1.1e5 - 7.1e2 - 4.1e2 &1.4e5 - 8.7e2 - 4.2e2\\
		\hline
	\end{tabular}
	\caption{Computational times [s] for \New - \NlinRAS - \NlinRASc, and different subdomain decompositions.}
	\label{tab:MONO3}
\end{table}
These also show the benefit of our continuation and preconditioning strategies.

\subsection{RASPEN}\label{sec:num:RASPEN}
Here, we focus on our strategies based on the RASPEN approach, and we present corresponding results of numerical experiments to assess the performance of \RASPEN{}, and \RASPENc. As in Section \ref{sec:num:Mono}, we first set a $2 \times 2$ subdomain decomposition and report in Table~\ref{tab:RASPEN1} number of outer RASPEN iterations, average number of parallel inner (subdomain) iterations, (average) number of GMRES iterations, and computational times (in seconds).
\setlength\tabcolsep{0.4mm}
\begin{table}[h!]
	\centering
	
	\begin{tabular}{|l|| c|c|c|c|c|}
		\hline
		$\varepsilon_{\min}$ & 1 & 1e-5& 1e-10 &  1e-15  \\
		\hline\hline
		\RASPEN &  3 - 6 - 33 - 174 &3 - 14 - 35 - 362 &3 - 15 - 34 - 389 &3 - 15 - 34 - 381 \\ 
		\RASPENc & 3 - 6 - 33 - 176 &3 - 5 - 35 - 161 &3 - 7 - 34 - 213 &3 - 8 - 34 - 231 \\ 
		\hline 
	\end{tabular}
	
	\caption{Outer iterations - average parallel inner iterations - average outer GMRES iterations - computational times [s] for different values of $\varepsilon=\varepsilon_{\min}$ and a $2\times2$ decomposition .}
	\label{tab:RASPEN1}
\end{table}
The results of Table~\ref{tab:RASPEN1} show clearly the benefit of using the continuation strategy in the first inner iteration. 
While the number of outer iterations is essentially constant (equal to 3), the number of parallel inner iterations is reduced by a factor of 2 when the continuation is used. 
The number of average outer GMRES iterations is stable in all cases and not influenced by the continuation.
Finally, the computational times are lower when the continuation is used, in agreement with the lower number of inner iterations. Therefore, according to Section \ref{sec:num:Mono} and Table \ref{tab:RASPEN1} the continuation strategy improves both the performance of monolithic Newton methods (with and without linear preconditioning) as well as the nonlinear preconditioned method due to the improvement in the inner Newton.

Next, we study the behavior of our numerical frameworks with respect to the number of subdomains and perform numerical experiments using the same settings of Section \ref{sec:num:Mono}.
Table~\ref{tab:RASPEN2} shows the number of outer iterations,
\setlength\tabcolsep{1.6mm}
\begin{table}[h!]
	\centering
	\begin{tabular}{|c||c|c|c|c|}
		\hline
		\# sub/ $\varepsilon_{\min}$ &1& 1e-5& 1e-10 & 1e-15  \\
		\hline\hline
		$2\times 2$&3 - 3&3 - 3&3 - 3&3 - 3 \\ 
		$2\times 3$&5 - 5&5 - 5&5 - 5&5 - 5 \\ 
		$2\times 4$&5 - 5&5 - 5&5 - 5&5 - 5 \\ 
		$2\times 5$&3 - 3&3 - 3&3 - 3&3 - 3 \\ 
		$2\times 6$&5 - 5&5 - 5&5 - 5&5 - 5 \\ 
		$2\times 7$&5 - 5&5 - 5&5 - 5&5 - 5 \\ 
		$2\times 8$&5 - 5&5 - 5&5 - 5&5 - 5 \\ 
		\hline
	\end{tabular}
	\caption{Outer iterations for \RASPEN{} - \RASPENc.}
	\label{tab:RASPEN2}
\end{table}
from which it is clear that all methods are robust against the number of subdomains and the regularization parameter $\varepsilon_{\min}$.
To further investigate the performances, we report in Table~\ref{tab:RASPEN3} the average number of inner iterations in dependence on the number of subdomains and the regularization parameter.
\setlength\tabcolsep{1.6mm}
\begin{table}[h!]
	\centering
	
	\begin{tabular}{|c||c|c|c|c|}
		\hline
		\# sub/ $\varepsilon_{\min}$ &1& 1e-5& 1e-10 & 1e-15  \\
		\hline\hline
		$2\times 2$&5 - 5&13 - 4&15 - 7&14 - 8 \\
		$2\times 3$&4 - 4&9 - 4&10 - 5&9 - 6 \\ 
		$2\times 4$&5 - 5&9 - 4&9 - 5&9 - 6 \\ 
		$2\times 5$&6 - 6&14 - 5&15 - 7&15 - 8 \\ 
		$2\times 6$&5 - 5&9 - 4&10 - 5&9 - 6 \\ 
		$2\times 7$&5 - 5&9 - 4&10 - 6&10 - 7 \\ 
		$2\times 8$&5 - 5&9 - 4&10 - 5&10 - 6 \\ 
		\hline
	\end{tabular}
	\caption{Average parallel inner iterations for \RASPEN{} - \RASPENc. }
	\label{tab:RASPEN3}
\end{table}
As before, one can observe the benefit of the continuation approach, resulting in a reduction in parallel iterations by up to half (for $\varepsilon_{\min}=10^{-5},10^{-10},10^{-15}$).
Table~\ref{tab:RASPEN4} shows the number of average GMRES iterations, which grow with increasing number of subdomains, but stay almost constant for decreasing $\varepsilon_{\min}$.
\setlength\tabcolsep{0.3mm}
\begin{table}[h!]
	\centering
	\begin{tabular}{|c||c|c|c|c|}
		\hline
		\#sub/$\varepsilon_{\min}$ &1& 1e-5& 1e-10 & 1e-15   \\
		\hline\hline
		$2\times 2$&29 - 29&28 - 28&27 - 27&27 - 27 \\ 
		$2\times 3$&34 - 34&38 - 38&39 - 39& 39 - 39 \\ 
		$2\times 4$&36 - 36&39 - 39&39 - 39&39 - 39 \\ 
		$2\times 5$&33 - 33&35 - 35&34 - 34&34 - 34 \\ 
		$2\times 6$&37 - 37&43 - 43&42 - 42&41 - 41 \\ 
		$2\times 7$&38 - 38&43 - 43&42 - 42&43 - 43 \\ 
		$2\times 8$&39 - 39&44 - 44&44 - 44&43 - 43 \\ 
		\hline
	\end{tabular}
	
	\caption{Average GMRES iterations for \RASPEN{} - \RASPENc.}
	\label{tab:RASPEN4}
\end{table}
As $\varepsilon_{\min}$ decreases, the advantage of the continuation strategy becomes evident. Additionally, the benefit of parallelization becomes apparent when more subdomains are used.
\setlength\tabcolsep{0.3mm}
\begin{table}[h!]
	\centering
	
	\begin{tabular}{|c||c|c|c|c|}
		\hline
		\#sub/$\varepsilon_{\min}$ &1& 1e-5& 1e-10 & 1e-15  \\
		\hline\hline
		$2\times 2$&335 - 331&762 - 294&865 - 433&811 - 486 \\ 
		$2\times 3$&353 - 354&674 - 355&703 - 444&683 - 464 \\ 
		$2\times 4$&297 - 300&506 - 293&523 - 337&518 - 383 \\ 
		$2\times 5$&174 - 176&362 - 161&389 - 213&381 - 231 \\ 
		$2\times 6$&230 - 234&385 - 224&394 - 264&378 - 277 \\ 
		$2\times 7$&224 - 221&354 - 218&367 - 257&363 - 277 \\ 
		$2\times 8$&221 - 220&331 - 221&338 - 238&335 - 260 \\ 	
		\hline
	\end{tabular}
	
	\caption{Computational times [s] for \RASPEN{} - \RASPENc{}.}
	\label{tab:RASPEN5}
\end{table}
\\
Finally, we compare all methods in Table~\ref{tab: comparison all methods 1 lvl n=450} for two subdomain decompositions $2\times2$ and $2\times 5$. In our experiments, we observe that nonlinearly preconditioned methods are more efficient than linearly preconditioned ones. Additionally, methods with continuation outperform those without in terms of computation time. Further, in Table \ref{tab: additional}, we compare \NlinRASc\ and \RASPENc{} for subdomain decompositions $s\times s$ for $s=2\ldots,8$. While the outer iterations (and parallel inner iterations) stay nearly constant for both methods, the outer GMRES iterations increase with an increasing number of subdomains. Also in most cases, \RASPENc{} is superior to \NlinRASc{} in terms of computation time.
\begin{table}[h!]
	\centering
	\begin{tabular}{|l||c|c|c|c|c|c|}
		\hline
		$2\times 2$& \New & \Newc & \NlinRAS &  \NlinRASc & \RASPEN & \RASPENc \\
		\hline\hline
		Outer it.& 40 & 23 & 40 & 23 & 3 & 3   \\ 
		Average outer GMRES it.& 1341 & 1128 & 33 & 31 & 27 & 27   \\ 
		Average parallel inner it.& - &  - &  - &  - & 14 & 8   \\ 
		Time [s] & 145009 & 51406 & 1179 & 640 & 811 & 486   \\
		\hline
	\end{tabular}
	\begin{tabular}{|l||c|c|c|c|c|c|}
		\hline
		$2\times 5$& \New & \Newc & \NlinRAS &  \NlinRASc & \RASPEN & \RASPENc \\
		\hline\hline
		Outer it.&	40 & 23 & 40 & 23 & 3 & 3   \\ 
		Average outer GMRES it.& 1341 & 1128 & 47 & 43 & 34 & 34   \\ 
		Average parallel inner it.&  - &  - &  - &  - & 15 & 8   \\ 
		Time [s] & 145009 & 51406 & 676 & 344 & 381 & 231  \\
		\hline
	\end{tabular}
	\label{tab: comparison all methods 1 lvl n=450}
	\caption{Comparison of all methods for a subdomain decomposition of $2\times 2$, $2\times 5$ and $\varepsilon_{\min} = 10^{-15}$.}
\end{table}
\begin{table}[h!]
	\centering
	\begin{tabular}{|l||c|c|c|c|c|c|c|}
		\hline
		\RASPENc & $2\times 2$ &$3\times 3$ & $4\times 4$ &  $5\times 5$ & $6\times 6$ &$7\times 7$& $8\times 8$ \\
		\hline\hline
		Outer it.& 5 & 6& 5 & 3 & 6 & 5& 5 \\ 
		Average outer GMRES it.& 27 & 41  & 45 &36  & 51 &  55 &57 \\ 
		Average parallel inner it.& 8 &  7 & 7  & 8 & 27 &7  & 7\\ 
		Time [s] & 489 & 468 &  276& 162 & 637 & 417 & 577 \\
		\hline
	\end{tabular}
	\begin{tabular}{|l||c|c|c|c|c|c|c|}
		\hline
		\NlinRASc & $2\times 2$ &$3\times 3$ & $4\times 4$ &  $5\times 5$ & $6\times 6$ &$7\times 7$ &$8\times 8$ \\
		\hline\hline
		Outer it.& 23 & 23 & 23 & 23 & 23 & 23& 23 \\ 
		Average outer GMRES it.& 31 & 45 & 57 & 46 & 56 &  64 &70 \\ 
		Average parallel inner it.& - &  - &  - &  - & - & - & - \\ 
		Time [s] & 501 & 403 & 421 &301  & 413 & 567 & 742 \\
		\hline
	\end{tabular}
	\caption{Comparison of \NlinRASc\  and \RASPENc\ for different decompositions and $\varepsilon_{\min} = 10^{-15}$.}
	\label{tab: additional}
\end{table}
%

\subsection{Weak scalability}
In this section, we present results on weak scalability for the methods \RASPENc{}, \RASPEN{}, \NlinRASc{} and \NlinRAS{}. 
We use a fixed subdomain size of $50^2$ gridpoints and increase the number of subdomains from $2\times 2$ to $11\times 11$ with the number of processors (Table \ref{tab:review_diff_scalability}).
%
%
\begin{table}[h!]
	\centering
	\begin{tabular}{|c||c|c|c|c|}
        \hline
        time [s]  & \RASPENc&  \RASPEN& \NlinRASc & \NlinRAS  \\
        \hline\hline    $2\times2$&7 - 3 - 14&9 - 3 - 14&41 - 27 - 17&47 - 31 - 17\\ 
$2\times2$&3 - 14 - 7&3 - 14 - 9&27 - 17 - 41&31 - 17 - 47\\ 
$3\times3$&5 - 21 - 17&5 - 21 - 20&27 - 27 - 60&37 - 28 - 83\\ 
$4\times4$&5 - 28 - 87&7 - 29 - 120&27 - 37 - 95&38 - 38 - 138\\ 
$5\times5$&3 - 29 - 40&3 - 29 - 50&27 - 38 - 138&38 - 40 - 205\\ 
$6\times6$&6 - 44 - 386&8 - 43 - 491&27 - 48 - 237&39 - 51 - 363\\ 
$7\times7$&5 - 49 - 254&5 - 49 - 272&27 - 59 - 405&40 - 62 - 647\\ 
$8\times8$&5 - 51 - 482&5 - 51 - 503&27 - 64 - 676&40 - 67 - 1056\\ 
$9\times9$&5 - 63 - 802&5 - 63 - 825&23 - 70 - 947&40 - 77 - 1845\\ 
$10\times10$&3 - 43 - 672&3 - 43 - 746&22 - 63 - 1111&40 - 69 - 2230\\ 
$11\times11$&5 - 77 - 1893&5 - 77 - 1913&22 - 83 - 2123&39 - 91 - 4151\\ 
        \hline
    \end{tabular}
	\caption{(Outer iterations - computational times [s] - Average GMRES iterations) for different algorithms and $\varepsilon_{\min}=10^{-15}$.}
	
	\label{tab:review_diff_scalability}
\end{table} 
We observe that the computation times increase slightly faster than linearly with the number of subdomains due to an increase in the number of GMRES iterations. 
This is due to the increasing number of subdomains.
A second level iteration could help mitigate this effect, but this is beyond the scope of this paper.
\subsection{Comparison of the algorithms for varying regularization and continuation parameters}
In this section, we analyze the robustness of the algorithms with respect to the regularization parameters $\mu$ and $\nu$, and the continuation parameters $\varepsilon_0$ and $\gamma$, see Problem \eqref{eq:P} and Algorithm \ref{alg:epsilonContinuation}.
\\
In Table \ref{tab:eps_rates_sensitivity}, we report the computation time and number of outer iterations of four algorithms (\RASPENc{}, \RASPEN{}, \NlinRASc{} and \NlinRAS{}).
The table is split into three blocks of two block rows each.
In the first block of Table \ref{tab:eps_rates_sensitivity}, we focus on a setting where the projection operator's nonlinearity is more pronounced, i.e., $\mu=1$ with small $L^2$-regularization $\nu=10^{-8}$.
For $\mu=1$, Newton's method without continuation needs around $86$ iterations to converge. 
For both linear and nonlinear preconditioned solvers, the continuation strategy reduces the number of outer and inner iterations, which is also observed in the computation times. 
Further, for this parameter setup, one observes that the \RASPEN{} methods are superior to the \NlinRAS{} methods. 
In this case, a smaller continuation rate $\gamma=2$ tends to give better results than $\gamma=5,10$. 
Furthermore, for the linear methods, a larger initial continuation parameter $\varepsilon_0=1$ works best, while for the nonlinear methods the choice of $\varepsilon_0=10^{-5}$ works best.\\
In the second block, we consider a small $L^1$-regularization $\mu=10^{-4}$ with small $L^2$-regularization $\nu=10^{-8}$. 
For $\mu=10^{-4}$, the $\varepsilon$-continuation does not generally improve the outer iterations and computation times. 
Only if the continuation parameters are well-chosen, one observes a small reduction in computation time and outer iterations.
Hence, in this case, only with appropriate tuning of the continuation parameters, one can observe computational improvements.
In the third block, we consider $\mu=1$ with larger $L^2$-regularization paramter $\nu=10^{-4}$. 
Here, one clearly sees that the continuation does not pay off, leading for all combinations of $(\varepsilon_0,\gamma)$ to an increase in outer and inner iterations and computation time. 
Also since the $L^2$-regularization is chosen to be relatively large, monolithic Newton converges fast in terms of outer iterations, and hence the linear preconditioner works better than the nonlinear one.

To summarize, the continuation strategy in combination with nonlinear preconditioning is most beneficial when $\mu$ is large relative to the $L^2$-regularization parameter.
\begin{table}[h!]
	\centering
	\begin{tabular}{|c||c|c|c|c|}
		\hline
		\multicolumn{4}{|c|}{ $\nu=10^{-8}$, $\mu=1$} \\ 
		\hline \hline
		time [s] & $\varepsilon_0=1$&  $\varepsilon_0=10^{-3}$ &   $\varepsilon_0=10^{-5}$   \\
		\hline
		$\gamma= 2$ &634 - 1148 - 856 - 1499 &342 - 1147 - 1177 - 1505 &269 - 1148 - 1247 - 1498 \\  
		$\gamma = 5$ &529 - 1149 - 905 - 1495 &491 - 1148 - 1183 - 1474 &513 - 1148 - 1222 - 1485 \\ 
		$\gamma= 10$ &1036 - 1147 - 1325 - 1491 &1014 - 1149 - 1438 - 1502 &1024 - 1147 - 1421 - 1510 \\
		\hline
		outer it. & $\varepsilon_0=1$&  $\varepsilon_0=10^{-3}$ &   $\varepsilon_0=10^{-5}$   \\
		\hline
		$\gamma= 2$ &3 - 3 - 52 - 86 &3 - 3 - 71 - 86 &3 - 3 - 74 - 86 \\  
		$\gamma = 5$ &3 - 3 - 53 - 86 &3 - 3 - 70 - 86 &3 - 3 - 71 - 86 \\ 
		$\gamma= 10$ &3 - 3 - 79 - 86 &3 - 3 - 82 - 86 &3 - 3 - 82 - 86 \\
		\hline
		\hline
		\multicolumn{4}{|c|}{$\nu=10^{-8}$, $\mu=10^{-4}$} \\ 
		\hline\hline
		time [s] & $\varepsilon_0=1$&  $\varepsilon_0=10^{-3}$ &   $\varepsilon_0=10^{-5}$   \\
		\hline
		$\gamma= 2$ &936 - 955 - 662 - 728 &936 - 955 - 661 - 715 &946 - 954 - 689 - 727 \\ 
		$\gamma = 5$ &957 - 955 - 717 - 722 &955 - 955 - 717 - 728 &956 - 956 - 743 - 717 \\ 
		$\gamma= 10$ &955 - 957 - 750 - 735 &955 - 956 - 747 - 735 &952 - 953 - 736 - 746 \\
		\hline
		outer it. & $\varepsilon_0=1$&  $\varepsilon_0=10^{-3}$ &   $\varepsilon_0=10^{-5}$   \\
		\hline
		$\gamma= 2$&4 - 4 - 75 - 78 &4 - 4 - 75 - 78 &4 - 4 - 77 - 78 \\ 
 
		$\gamma = 5$ &4 - 4 - 78 - 78 &4 - 4 - 78 - 78 &4 - 4 - 78 - 78 \\  
		$\gamma= 10$ &4 - 4 - 78 - 78 &4 - 4 - 78 - 78 &4 - 4 - 78 - 78 \\ 
		\hline\hline
		\multicolumn{4}{|c|}{$\nu=10^{-4}$, $\mu=1$} \\ 
		\hline \hline
		time [s] & $\varepsilon_0=1$&  $\varepsilon_0=10^{-3}$ &   $\varepsilon_0=10^{-5}$   \\
		\hline
		$\gamma= 2$ &540 - 203 - 1313 - 115 &331 - 207 - 620 - 117 &278 - 208 - 438 - 117 \\ 
		$\gamma = 5$&429 - 209 - 1086 - 117 &281 - 205 - 479 - 114 &246 - 208 - 351 - 116 \\ 
		$\gamma= 10$&357 - 209 - 888 - 113 &249 - 208 - 391 - 122 &227 - 212 - 302 - 117 \\ 
		\hline
		outer it. & $\varepsilon_0=1$&  $\varepsilon_0=10^{-3}$ &   $\varepsilon_0=10^{-5}$   \\
		\hline
		$\gamma= 2$&3 - 3 - 43 - 4 &3 - 3 - 20 - 4 &3 - 3 - 14 - 4 \\  
		$\gamma = 5$ &3 - 3 - 33 - 4 &3 - 3 - 15 - 4 &3 - 3 - 11 - 4 \\
		$\gamma= 10$ &3 - 3 - 26 - 4 &3 - 3 - 12 - 4 &3 - 3 - 9 - 4 \\ 
		\hline
	\end{tabular}
	\caption{Computational times and outer iterations for different parameters ($\gamma,\varepsilon_0,\mu,\nu$) and algorithms in order \RASPENc - \RASPEN - \NlinRASc - \NlinRAS.
	}
	\label{tab:eps_rates_sensitivity}
\end{table}
%


\section{Conclusion}\label{sec:concl}
In this contribution, we considered smooth approximations of optimality systems for $L^1$-regularized, semilinear optimal control problems. On a theoretical level, we established the solvability of the smoothed system and proved the convergence of the solution towards the solution of the nonsmooth system with convergence order. These considerations gave rise to a continuation approach that was combined with both linear and nonlinear RAS preconditioned Newton methods. The numerical experiments demonstrated, on the one hand, the efficiency of the continuation approach for both linear and nonlinear preconditioning, and on the other hand, the potential advantage of using nonlinear preconditioned methods over linear ones in situations involving significant nonlinearity and small $L^2$-regularization. In the future, one could consider extending the approach by suitable coarse correction strategies.
%

\section*{Declarations}
\subsection*{Conflict of interest}
The authors declare no conflict of interest.

\subsection*{Data availability statements}
The data is made available upon request.

\subsection*{Funding}
The authors did not receive support from any organization for the submitted work.

\bibliography{biblio}

\end{document}